\documentclass[12pt]{amsart}
\usepackage{a4wide,enumerate,color}
\usepackage{enumitem}
\allowdisplaybreaks

\let\pa\partial  
\let\na\nabla  
\let\eps\varepsilon  
\newcommand{\dd}[1]{\operatorname{d}\!#1}
\newcommand{\N}{{\mathbb N}}  
\newcommand{\R}{{\mathbb R}} 
\newcommand{\diver}{\operatorname{div}}  
\newcommand{\F}{{\mathcal F}}

\newtheorem{theorem}{Theorem}   
\newtheorem{lemma}[theorem]{Lemma}   
\newtheorem{proposition}[theorem]{Proposition}   
\newtheorem{remark}[theorem]{Remark}

\begin{document}  

\title[A fractional drift-diffusion-Poisson system]{Large-time asymptotics
of a fractional drift-diffusion-Poisson system via the entropy method} 

\author[F. Achleitner]{Franz Achleitner}
\address{Faculty of Mathematics, University of Vienna, 
Oskar-Morgenstern-Platz 1, 1090 Wien, Austria}
\email{franz.achleitner@univie.ac.at}

\author[A. J\"ungel]{Ansgar J\"ungel}
\address{Institute for Analysis and Scientific Computing, Vienna University of  
	Technology, Wiedner Hauptstra\ss e 8--10, 1040 Wien, Austria}
\email{juengel@tuwien.ac.at} 

\author[M. Yamamoto]{Masakazu Yamamoto}
\address{Graduate School of Science and Technology, Niigata University, Niigata, 
950-2181, Japan}
\email{masakazu@eng.niigata-u.ac.jp}

\date{\today}

\thanks{
The first and second author acknowledge support from the 
Austrian Science Fund (FWF), grant F65. 
The second author is supported by the Austrian Science Fund (FWF), 
grants P27352, P30000, and W1245.
The last author is partially supported by the Japan Society for the Promotion of 
Science (JSPS), Grant-in-aid for Early-Career Scientists (B) 15K17566}

\begin{abstract}
The self-similar asymptotics for solutions to the drift-diffusion equation with
fractional dissipation, coupled to the Poisson equation, is analyzed in the whole space.
It is shown that in the subcritical and supercritical cases, 
the solutions converge to the 
fractional heat kernel with algebraic rate. The proof is based on the entropy
method and leads to a decay rate in the $L^1(\R^d)$ norm. The technique
is applied to other semilinear equations with fractional dissipation.
\end{abstract}

 

\maketitle


\section{Introduction}

In this paper, we investigate the large-time behavior of solutions to a drift-diffusion
equation with fractional diffusion, coupled self-consistently to the Poisson
equation. Such models describe the evolution of particles in a fluid under the
influence of an acceleration field. The particle density $\rho(x,t)$ and
potential $\psi(x,t)$ satisfy the equations
\begin{equation}\label{1.eq}
  \pa_t\rho + (-\Delta)^{\theta/2}\rho = \diver(\rho\na\psi), \quad
	-\Delta\psi = \rho\quad\mbox{in }\R^d,\ t>0,
\end{equation}
with initial condition
\begin{equation}\label{1.ic}
  \rho(\cdot,0) = \rho_0\quad\mbox{in }\R^d.
\end{equation}
The fractional Laplacian $(-\Delta)^{\theta/2}$ is defined by
$(-\Delta)^{\theta/2}\rho = \F^{-1}[|\xi|^\theta\F[\rho]]$, where $\F$ is the
Fourier transform, $\F^{-1}$ its inverse, and $\theta>0$. When $\theta=2$,
we recover the standard drift-diffusion-Poisson system arising in semiconductor
theory and plasma physics \cite{Jue09}. 
Drift-diffusion-type equations with $\theta<2$ were
proposed to describe chemotaxis of biological cells whose behavior is not governed
by Brownian motion \cite{Esc06}.
For given acceleration field $\na\psi$ and $1<\theta<2$, 
the first equation in \eqref{1.eq} was derived from the Boltzmann equation
by Aceves-Sanchez and Mellet \cite{AcMe17}, based on the moment mehod of
Mellet \cite{Mel10}. For the convenience of the reader,
we present a formal derivation of the coupled system in Appendix \ref{sec.deriv}.
Note, however, that we consider equations \eqref{1.eq} in the range $0<\theta\le 2$.

The aim of this paper is to compute the decay rate of the solution
to \eqref{1.eq} to self-similarity in the $L^1(\R^d)$ norm using the entropy
method. In previous works \cite{BKW01,LRZ10,YaSu16}, the self-similar asymptotics
of various model variations were shown in $L^p(\R^d)$ norms but the decay rate
is zero when $p=1$. If additionally the first moment exists, i.e.\ if
$|x|\rho_0\in L^1(\R^d)$, the decay rate of the self-similar asymptotics in the 
$L^1(\R^d)$ norm is $1/\theta$, which is optimal \cite[Lemma 5.1]{OgYa09}. 
The entropy method provides an alternative way to analyze the self-similar asymptotics
in the $L^1(\R^d)$ norm. Moreover, its strength is its robustness, i.e., the method can
be easily applied to other semilinear equations with fractional dissipation. We give
two examples in Section \ref{sec.other}.

Before stating our main result and the key ideas of the technique, we review
the state of the art for drift-diffusion equations. The {\em global existence of
solutions} to the drift-diffusion-Poisson system was shown 
for $\theta=2$ in \cite{KuOg08}, 
for the subcritical case $1<\theta\le 2$ in \cite{LRZ10,OgYa09}, 
and for the supercritical case $0<\theta<1$ in \cite{LRZ10,SYK15}.
For suitable initial data, the solution to \eqref{1.eq} satisfies
\begin{align}
  & \rho\in C^0([0,\infty);L^1(\R^d)\cap L^\infty(\R^d)), \quad 
	\rho(t)\ge 0\quad\mbox{in }\R^d,\ t>0, \label{1.p1} \\
	& \|\rho(t)\|_{L^1(\R^d)} = \|\rho_0\|_{L^1(\R^d)}, \quad
	\|\rho(t)\|_{L^p(\R^d)} \le C(1+t)^{-\frac{d}{\theta}(1-\frac{1}{p})}, \quad
	1\le p\le\infty; \label{1.p2}
\end{align}
see \cite[Theorem~1.1]{OgYa09} for $1<\theta\leq 2$, \cite[Theorem 1]{YKS14} for
$\theta=1$ and $d\ge 3$, and \cite[Theorem 1.7]{LRZ10} for $0<\theta<1$ and $d=2$.
The existence result of \cite{OgYa09} for the bipolar drift-diffusion system
was extended by Granero-Belinch\'on \cite{Gra16} 
by allowing for different fractional exponents.
The existence of solutions in Besov spaces was proved for $1<\theta<2d$ and 
$d\geq 2$ in~\cite{ZhLi14}, and for $0<\theta\le 1$ and $d\geq 3$ in~\cite{SYK15}.

The self-similar asymptotics of the {\em fractional heat equation} was studied by
Vazquez \cite[Theorem 3.2]{Vaz17}, 
showing that if $0<\theta<2$ and the initial datum satisfies
$(1+|x|)\rho_0\in L^1(\R^d)$, we have 
$\|\rho(t)-MG_\theta(t)\|_{L^1(\R^d)}\le Ct^{-1/\theta}$ with optimal rate,
where $M=\int_{\R^d}\rho_0\dd{x}$ is the initial mass and
$G_\theta$ is the fundamental solution to the fractional heat equation
(see Section \ref{sec.frac}).
Exploiting the self-similar structure, this proves the exponential decay 
with rate $1/\theta$ to the fractional Fokker-Planck equation with quadratic potential. 
The exponential decay in $L^1$ spaces with weight $1+|x|^k$ and $k<\theta$
was proved by Tristani \cite{Tri15}.
The fractional Laplacian can be replaced by more general L\'evy operators,
and the large-time asymptotics of so-called L\'evy-Fokker-Planck equations
were investigated by Biler and Karch \cite{BiKa03} as well as Gentil and
Imbert \cite{GeIm08}. 

The large-time behavior of solutions to {\em drift-diffusion-Poisson systems} with
$d=2$ and $\theta=2$ was studied by Nagai \cite{Nag11}, showing the decay of the solutions
to zero. A similar result for $0<\theta<2$ was proven by Li et al.\ \cite{LRZ10}.
The self-similar asymptotics in $L^p(\R^d)$ with $1\le p\le\infty$
was shown in \cite{KoKa08} for $\theta=2$ and in \cite{OgYa09} for $1<\theta<2$. 
In the latter reference, also the decay of the first-order asymptotic expansion 
of the solutions was computed. Higher-order expansions were studied for $1<\theta\le 2$
in \cite{Yam12}, for $0<\theta\le 1$ in \cite{YaSu16}, and
for the critical case $\theta=1$ in \cite{YaSu16a}.
However, in most of these references, the decay rate for $p=1$ is zero.

The exponential decay in the relative entropy for solutions to L\'evy-Fokker-Planck
equations was proved in \cite{BiKa03, GeIm08}. Via the Csisz\'ar-Kullback inequality 
(see, e.g., \cite{AMTU01}),
this implies decay in the $L^1(\R^d)$ norm. In fact, we are using the techniques
of \cite{GeIm08}, combined with tools from harmonic analysis and semigroup theory, to
achieve self-similar decay of solutions to~\eqref{1.eq}. 

Our main result is as follows.

\begin{theorem}\label{thm.main}
Let $d\ge 2$ and $0<\theta\le 2$ with $\theta<d/2$.
Let $\rho_0\in L^1(\R^d)\cap L^\infty(\R^d)$ be nonnegative such that 
$|x|^d\rho_0\in L^q(\R^d)$ for some $q > d/\theta$. Furthermore,
let $\rho$ be a solution to \eqref{1.eq} satisfying \eqref{1.p1}-\eqref{1.p2}.
Then, for all $t>0$,
\begin{equation}\label{1.decay}
  \|\rho(t)-MG_\theta(t)\|_{L^1(\R^d)} \le C(1+\theta t)^{-1/2},
\end{equation}
where $C>0$ depends on $\rho_0$ and $\theta$, $M=\int_{\R^d}\rho_0\dd{x}$,
and $G_\theta(x,t)=\F^{-1}[e^{-|\xi|^\theta t}](x)$ is the fundamental solution
to $\pa_t u + (-\Delta)^{\theta/2}u=0$ in $\R^d$. 
\end{theorem}

Let us comment on the theorem. 
We need the finiteness of the moment $|x|^d\rho_0$ in $L^{q}(\R^d)$
to guarantee the well-posedness of the entropy functional; 
see Lemma~\ref{lem.estrho} and Step 4 in the proof of Theorem~\ref{thm.main}.
The lower bound of $q$ is rather natural
since $|x|^d G_\theta\in L^q(\R^d)$ if and only if $q>d/\theta$.
If $\rho_0$ grows like $(1+r)^{-\alpha}$ for large values of the radius $r=|x|$,
the condition $|x|^d\rho_0\in L^q(\R^d)$ for large values of $q$ is only slightly 
stronger than $\rho_0\in L^1(\R^d)$. 
Indeed, in the latter case, we need $\alpha>d$, while $\alpha>(1+1/q)d$ is required in 
the former case.
The condition $\theta<d/2$ is needed to estimate the nonlinear drift term;
see the proof of Lemma \ref{lem.decayu}.
The decay rate is not optimal. This may be due to the fact the first moment of
the fundamental solution $G_\theta$ is finite for all $1<\theta<2$ but not for
$0<\theta<1$. The derivation of \eqref{1.eq} leads to a drift term 
$\diver(\rho D\na\psi)$ involving the drift matrix $D\in\R^{d\times d}$. 
We explain in Appendix \ref{sec.drift} that we are able to treat 
only the case when $D$ equals the identity matrix (times a factor and up to adding
a skew-symmetric matrix).

As already mentioned, 
the idea of the proof is to employ the entropy method, originally developed for
stochastic processes by Bakry and Emery \cite{BaEm85} and later extended to
linear and nonlinear diffusion equations (see, e.g., \cite{AMTU01,CJMTU01}). 
First, we reformulate \eqref{1.eq} in terms of the rescaled function
$$
  u(x,t) = \frac{e^{dt}}{M}\rho\bigg(e^t x,\frac{e^{\theta t}-1}{\theta}\bigg),
	\quad x\in\R^d,\ t>0.
$$
This function solves a drift-diffusion-Poisson system with the confinement potential
$V(x)=\frac12|x|^2$ and with the nonnegative steady state $u_\infty=G_\theta(1/\theta)$.
Next, we show that the relative entropy
$$
  E_p\bigg[\frac{u}{u_\infty}\bigg] 
	= \int_{\R^d}\bigg(\frac{u}{u_\infty}\bigg)^p u_\infty \dd{x}
	- \bigg(\int_{\R^d}u \dd{x}\bigg)^p
$$
satisfies the inequality
$$
  \frac{\dd{E_p}}{\dd{t}} \le -(\theta+\sigma(t))E_p + \sigma(t)
$$
for some function $\sigma(t)$ which comes from the drift term involving $\psi$
and which decays to zero exponentially fast. 
For this result, we need some results for L\'evy operators due to \cite{GeIm08}
and a modified logarithmic Sobolev inequality due to Wu~\cite{Wu00} and 
Chafa\"i \cite{Cha04}.
By Gronwall's lemma, we conclude the exponential convergence of 
$t\mapsto E_p[u(t)/u_\infty]$. Then the Csisz\'ar-Kullback inequality implies that
$u(t)-u_\infty$ converges exponentially fast in the $L^1(\R^d)$ norm. 
Finally, scaling back to the original variable, we deduce the algebraic decay for
$\rho(t)-MG_\theta(t)$ in the $L^1(\R^d)$ norm.

The strength of the entropy method is that it is quite robust. It can be applied
to other equations with fractional dissipation, at least if the regularity
and decay properties \eqref{1.p1}-\eqref{1.p2} hold. As examples, we consider
the two-dimensional quasi-geostrophic equation and a generalized fractional
Burgers equation in one space dimension; see Section \ref{sec.other}.

The paper is organized as follows. We summarize some results on the fractional
heat equation and L\'evy operators in Section \ref{sec.pre}. The proof
of Theorem \ref{thm.main} is given in Section \ref{sec.proof}. In Section 
\ref{sec.other}, the entropy method is applied to other equations. 
In Appendix \ref{sec.estrho}, a weighted $L^q(\R^d)$ estimate for $\rho$ is shown.
Appendix \ref{sec.deriv} is concerned with the formal derivation of \eqref{1.eq},
summarizing the ideas of \cite{AcMe17}. Finally, we explain in Appendix
\ref{sec.drift} that we can only treat drift matrices of the form $D=aI+B$,
where $a>0$, $I$ is the unit matrix, and $B$ is a skew-symmetric matrix.


\section{Preliminaries}\label{sec.pre}

We need two ingredients for our analysis. 
The first one are properties of the fractional heat kernel,
 the second one is a modified logarithmic Sobolev inequality
 which relates the entropy $E_p$ and the entropy production $-\dd{E_p}/\dd{t}$. 
For the convenience of the reader, we collect the needed results.

\subsection{Fractional heat equation}\label{sec.frac}

We consider the fractional heat equation
\begin{equation}\label{a.fhe}
  \pa_t u + (-\Delta)^{\theta/2}u = 0\quad\mbox{in }\R^d, \quad t>0,
\end{equation}
for $0<\theta\leq 2$.
The operator $(-\Delta)^{\theta/2}$ is defined
for $\theta>0$ and functions $u$ in the Schwartz space of rapidly decaying 
functions on $\R^d$ 
via Fourier transformation by $(-\Delta)^{\theta/2}u=\F^{-1}[|\xi|^\theta\F[u]]$,
where $\F[u](x)=(2\pi)^{-d/2}\int_{\R^d}u(\xi)e^{-ix\cdot\xi}\dd{\xi}$ is the Fourier 
transform of $u$, and $\F^{-1}$ is its inverse. 
In the limit $\theta\to 2$,
the Laplace operator $-\Delta$ is recovered, but for $\theta\neq 2$, 
$(-\Delta)^{\theta/2}$
is a nonlocal operator. The fractional Laplacian can be expressed as the
singular integral operator \cite{DrIm06}
$$
  (-\Delta)^{\theta/2}u = -c_{d,\theta}\,\mathrm{p.\,v.}\int_{\R^d}
 	\frac{u(x+y)-u(x)}{|y|^{d+\theta}}\dd{y}, \quad
	c_{d,\theta} = \frac{\theta}{2\pi^{d/2+\theta}}
	\frac{\Gamma(\frac12(d+\theta))}{\Gamma(\frac12(2-\theta))},
$$
where $u$ is a suitable function, p.v.\  denotes the Cauchy principal value,
and $\Gamma$ is the Gamma function. The principal value can be avoided for
$0<\theta<2$, and the integral becomes a standard one when $0<\theta<1$
\cite[Theorem~1]{DrIm06}.

The fundamental solution $G_\theta$ of the fractional heat equation~\eqref{a.fhe} 
is given by
\begin{equation}\label{a.g}
  G_\theta(x,t) = \F^{-1}[e^{-t|\xi|^\theta}](x), \quad x\in\R^d,\ t>0.
\end{equation}
We recall some qualitative properties of $G_\theta$:
\begin{itemize}
\item Normalization: $\int_{\R^d}G_\theta(x,t)\dd{x}=1$ for $t>0$.
\item Self-similar form: A direct computation shows that
\begin{equation}\label{a.self}
  \lambda^d G_\theta(\lambda x,\lambda^\theta t) = G_\theta(x,t)
	\quad\mbox{for all }x\in\R^d,\ t>0,\ \lambda>0.
\end{equation}
\item Pointwise estimates \cite[Theorem 7.3.1]{Kol11}: For any $K>0$, there
exists a constant $C>1$, which can be chosen uniformly for $\theta$ from any 
compact interval in $(0,2)$, such that
\begin{equation}\label{a.asymp}
\begin{aligned}
  \frac{1}{Ct^{d/\theta}} \le G_\theta(x,t) \le \frac{C}{t^{d/\theta}}
	&\quad\mbox{if }|x|\le Kt^{1/\theta},\ t>0, \\
	\frac{t}{C|x|^{d+\theta}} \le G_\theta(x,t) \le \frac{Ct}{|x|^{d+\theta}}
	&\quad\mbox{if }|x|\ge Kt^{1/\theta},\ t>0.
\end{aligned}
\end{equation}
\item Gradient estimate \cite[Theorem 7.3.2]{Kol11}: There exists $C>0$ such that
\begin{equation}\label{a.grad}
  |\na G_\theta(x,t)| \le C\min\big\{t^{-1/\theta},|x|^{-1}\big\}G_\theta(x,t),
	\quad x\in\R^d,\ t>0,
\end{equation}
holds uniformly for $\theta$ from any compact interval in $(0,2)$.
\end{itemize}

Here and in the following, $C>0$ denotes a generic constant independent of $x$ and $t$.



\subsection{L\'evy operators}\label{sec.levy}

We recall some results for L\'evy operators. A L\'evy operator $\mathcal{I}$ is
the infinitesimal generator associated with a L\'evy process. According to
the L\'evy-Khinchine formula \cite{Sat13}, it can be written in the form
$$
  \mathcal{I}[u] = \diver(\sigma\na u) - b\cdot\na u + \int_{\R^d}
	\big(u(x+y)-u(x)-\na u(x)\cdot y h(y)\big)\nu(\dd{y}),
$$
defined for suitable functions $u$, where $\sigma$ is a symmetric positive semidefinite
$d\times d$ matrix, $b\in\R^d$, $h(y)=(1+|y|^2)^{-1}$ is a truncation function, 
and $\nu$ denotes a nonnegative singular measure on $\R^d$ satisfying
\[ 
  \nu(\{0\})=0 \quad \text{and} \quad  
  \lim_{\eps\to 0}\int_{\{|y|>\eps\}}\min\{1,|y|^2\}\nu(\dd{y})<\infty;
\]
see, e.g., \cite[Section 1]{GeIm08}. 
Fractional Laplacians $-(-\Delta)^{\theta/2}$ for $0<\theta<2$ are L\'evy operators 
 with $\sigma=0$, $b=0$, and L\'evy measure $\nu(\dd{y})=\dd{y}/|y|^{d+\theta}$.
For more details on L\'evy operators, we refer to \cite{App04,Sat13}.

Let $\Phi:\R_+\to\R$ be a smooth convex function
 and $u_\infty$ a positive function satisfying $\int_{\R^d}u_\infty \dd{x}=1$.
The $\Phi$-entropy is defined by
\[
  \mathrm{Ent}_{u_\infty}^\Phi(f) = \int_{\R^d}\Phi(f)u_\infty \dd{x}
	- \Phi\bigg(\int_{\R^d}fu_\infty \dd{x}\bigg),
\]
for non-negative functions $f$. Observe that
Jensen's inequality implies that $\mathrm{Ent}_{u_\infty}^\Phi(f)\geq 0$.

We need two results for the $\Phi$-entropy.
In the following we only consider the family of smooth convex functions
 $\Phi_p:\R_+\to\R$, $s\mapsto s^p$, with $p\in(1,2]$.
However, the results hold for any smooth convex function $\Phi$
such that the mappings $(a,b)\mapsto D_\Phi(a+b,b)$ and 
$(a,y)\mapsto\Phi''(a)y\cdot(\sigma y)$
 are convex on $\{a+b\geq 0,\ b\geq 0\}$ and $\R_+\times\R^{2d}$, respectively.
Here, $D_\Phi(a,b):=\Phi(a)-\Phi(b)-\Phi'(b)(a-b)$ denotes the Bregman distance.
 
The first result is a modified logarithmic Sobolev inequality.

\begin{proposition} \label{prp.logsob}
Let $p\in(1,2]$ and consider $\Phi:\R_+\to\R$, $s\mapsto s^p$.
If $u_\infty$ is the density of an infinitely divisible probability measure
 then for all smooth positive functions $f$
 \begin{equation}\label{a.logsob}
  \mathrm{Ent}_{u_\infty}^\Phi(f) 
	\le \int_{\R^d}\Phi''(f)\na f\cdot(\sigma_\infty\na f) u_\infty \dd{x}
	+ \int_{\R^d}\int_{\R^d}D_\Phi(f(x),f(x+y))\nu_\infty(\dd{y})u_\infty \dd{x},
 \end{equation}
 where $\nu_\infty$ and $\sigma_\infty$ are the L\'evy measure
 and the diffusion matrix associated with $u_\infty$, respectively.
\end{proposition}
This result was first proved for particular cases by Wu~\cite{Wu00}
 and generalized by Chafa\"i~\cite{Cha04}, see also \cite[Theorem 2]{GeIm08}. 
 
The second result is a formula for the time derivative of the $\Phi$-entropy
 along solutions to the L\'evy-Fokker-Planck equation
 \begin{equation} \label{a.LFPe}
  \pa_t u = \mathcal{I}[u] + \diver(x u)\quad\mbox{for }x\in\R^d, \quad t>0,
 \end{equation}
 see~\cite[Proposition 1]{GeIm08}.
In fact, we need this result only for the special case
 of fractional Laplacians $\mathcal{I}=-(-\Delta)^{\theta/2}$ with $0<\theta<2$.

\begin{proposition}[{\cite[Theorem 1]{GeIm08}}] \label{prp.GeIm08.thm1}
Let $p\in(1,2]$ and consider $\Phi:\R_+\to\R$, $s\mapsto s^p$.
Consider the L\'evy-Fokker-Planck equation~\eqref{a.LFPe} 
 for fractional Laplacians $\mathcal{I}=-(-\Delta)^{\theta/2}$ with $0<\theta<2$
and stationary solution $u_\infty(x)=G_\theta(x,1/\theta)$. 
If $u_0$ is a nonnegative function with 
$\mathrm{Ent}_{u_\infty}^\Phi(u_0/u_\infty)<\infty$,
then the solution $u$ of~\eqref{a.LFPe} with initial datum~$u_0$ satisfies 
for all $t\geq 0$,
$$ 
  \mathrm{Ent}_{u_\infty}^\Phi(u(t)/u_\infty)
  \leq e^{-\theta t} \mathrm{Ent}_{u_\infty}^\Phi(u_0/u_\infty).
$$
\end{proposition}
\begin{proof}
We claim that
$u_\infty(x)=G_\theta(x,1/\theta)$ is the stationary solution of~\eqref{a.LFPe},
i.e. $\mathcal{I}[u_\infty] + \diver(xu_\infty)=0$. 
Indeed, the Fourier transform of this equation is
$|\xi|^\theta\widehat u_\infty + \xi\cdot\na\widehat u_\infty = 0$,
whose solution is given by $\widehat u_\infty(\xi) = e^{-|\xi|^\theta/\theta}$,
and the claim follows from the definition of the fractional heat kernel.
Moreover, $u_\infty(x)=G_\theta(x,1/\theta)$ is the density of an infinitely 
divisible probability measure 
with $\sigma_\infty=0$ and L\'evy measure $\nu_\infty(\dd{y})=\dd{y}/\theta 
|y|^{d+\theta}$. Then, for $v=u/u_\infty$,
\begin{align}
  \frac{\dd{}}{\dd{t}}\textrm{Ent}_{u_\infty}^\Phi(v)
	&= -\int_{\R^d}\int_{\R^d}D_\Phi(v(x),v(x-y))\nu(\dd{y})u_\infty(x) \dd{x} \nonumber \\
	&= -\int_{\R^d}\int_{\R^d}D_\Phi(v(x),v(x+y))\nu(\dd{y})u_\infty(x) \dd{x} 
	\label{a.dEdt} \\
	&\le -\theta \textrm{Ent}_{u_\infty}^\Phi(v). \nonumber
\end{align}
The first identity follows from \cite[Proposition 1]{GeIm08},
in the second one we substituted $y\mapsto -y$ (and used $\nu(-\dd{y})=\nu(\dd{y})$), 
and the inequality is a consequence of the logarithmic Sobolev inequality 
\eqref{a.logsob}. The factor $\theta$ comes from
the relation between the L\'evy measures $\nu$ and $\nu_\infty$,
see the comment after Theorem 1 in \cite{GeIm08}.
The final statement follows from Gronwall's inequality.
\end{proof}

\begin{lemma}\label{lem.heat}
Let $0<\theta\le 2$ and $G_\theta$ be the fundamental solution to the
fractional heat equation \eqref{a.fhe}. Then for all $s>0$,
$$
  \big\|G_\theta(s+1/\theta)-G_\theta(s)\big\|_{L^1(\R^d)}
	\le C(1+\theta s)^{-1/2}.
$$
\end{lemma}

\begin{proof}
A direct computation shows that the 
function $U(x,t)=e^{dt}G_\theta(e^tx,(e^{\theta t}-1)/\theta)$
solves the fractional Fokker-Planck equation
$$
  \pa_t U + (-\Delta)^{\theta/2}U = \diver(xU)\quad\mbox{in }\R^d,
$$
and $U_\infty=G_\theta(1/\theta)$ is the stationary solution.
Setting $s=(e^{\theta t}-1)/\theta$ and using the substitution $y=e^{-t}x$
and the self-similar form \eqref{a.self} with $\lambda=e^t$, we find that
\begin{align*}
  \big\|G_\theta(s+1/\theta)-G_\theta(s)\big\|_{L^1(\R^d)}
	&= \int_{\R^d}\big|G_\theta(x,e^{\theta t}/\theta)
	-G_\theta(x,(e^{\theta t}-1)/\theta)\big|\dd{x} \\
	&= \int_{\R^d}\big|G_\theta(e^t y,e^{\theta t}/\theta)
	-G_\theta(e^t y,(e^{\theta t}-1)/\theta)\big|e^{dt} \dd{y} \\
	&= \int_{\R^d}\big|G_\theta(y,1/\theta)-e^{dt}G_\theta(e^t y,(e^{\theta t}-1)/\theta)
	\big|\dd{y} \\
	&= \|U_\infty-U(t)\|_{L^1(\R^d)}.
\end{align*}
Using Proposition~\ref{prp.GeIm08.thm1} and the Csisz\'ar-Kullback inequality
 (see \cite{AMTU01} or \cite[Theorem A.3]{Jue16})
$$
  \|U_\infty-U(t)\|_{L^1(\R^d)} \le Ce^{-\theta(t-t_0)/2} 
	= Ce^{-\theta t_0/2}(1+\theta s)^{-1/2},
$$
and the constant $C>0$ depends on $\textrm{Ent}^\Phi_{U_\infty}[U(t_0)/U_\infty]$
for some small $t_0>0$. 
This finishes the proof. 
\end{proof}

\begin{remark}\rm
A direct estimate allows us to prove 
\begin{equation}\label{rem.1}
  \big\|G_\theta(s+1/\theta)-G_\theta(s)\big\|_{L^1(\R^d)}
	\le C(1+s)^{-1}.
\end{equation}
However, the constant $C$ is of order $1/\theta^2$,
hence, it is not uniformly bounded for $\theta\in(0,2)$.
In order to show \eqref{rem.1}, we observe that the integrand of
$$
  G_\theta(s+1/\theta)-G_\theta(s) = \frac{1}{\theta}\int_0^1
	\pa_t G_\theta(s+\lambda/\theta)d\lambda
$$
can be written for $t=s+\lambda/\theta$ as
\begin{align*}
  \pa_t G_\theta(x,t) &= \pa_t\big(t^{-d/\theta} G_\theta(t^{-1/\theta}x,1)\big) \\
	&= -\frac{1}{\theta t}\Big(d\, t^{-d/\theta}G_\theta(t^{-1/\theta}x,1)
	- t^{-d/\theta}(t^{-1/\theta}x)\cdot\na G_\theta(t^{-1/\theta}x,1)\Big),
\end{align*}
using the self-similar property \eqref{a.self}. By the pointwise bounds
\eqref{a.asymp} and \eqref{a.grad}, we find that
$$
  \|G_\theta(s+1/\theta)-G_\theta(s)\|_{L^1(\R^d)}
	\le C\int_0^1(s+\lambda/\theta)^{-1}d\lambda\le C\int_0^1 s^{-1}d\lambda
  = Cs^{-1}.
$$
Taking into account the bound
$$
  \|G_\theta(s+1/\theta)-G_\theta(s)\|_{L^1(\R^d)}
	\le \|G_\theta(s+1/\theta)\|_{L^1(\R^d)} + \|G_\theta(s)\|_{L^1(\R^d)} \le 2,
$$
the claim \eqref{rem.1} follows.
\end{remark}


\section{Proof of the main result}\label{sec.proof}

We split the proof of Theorem \ref{thm.main} into several steps.

{\em Step 1: Time-dependent rescaling of the equation.} 
Let $M=\int_{\R^d}\rho_0\dd{x}>0$. We introduce the rescaled function
\begin{equation}\label{2.defu}
  u(x,t) = \frac{e^{dt}}{M}\rho\bigg(e^t x,\frac{e^{\theta t}-1}{\theta}\bigg),
	\quad x\in\R^d,\ t>0.
\end{equation}

\begin{lemma}\label{lem.scal}
The function $u$ solves the confined drift-diffusion-Poisson system
\begin{equation}\label{2.equ}
\begin{aligned}
  \pa_t u + (-\Delta)^{\theta/2}u = \diver(u\na V) + Me^{-(d-\theta)t}\diver(u\na\phi)
	&\quad\mbox{in }\R^d,\ t>0, \\
	-\Delta\phi = u, \quad u(\cdot,0)=u_0 &\quad\mbox{in }\R^d,
\end{aligned}
\end{equation}
where $V(x)=\frac12|x|^2$ and $u_0=\rho_0/M$.
\end{lemma}

\begin{proof}
First, we observe that $u$ fulfills
\begin{equation}\label{2.aux1}
  \pa_t u - \diver(u\na V) = \frac{1}{M}e^{(d+\theta)t}\pa_t\rho
	\bigg(e^t x,\frac{e^{\theta t}-1}{\theta}\bigg)
\end{equation}
and, substituting $y=e^t x$,
\begin{align*}
  \widehat u(\xi,t) &= \frac{1}{M}(2\pi)^{-d/2}\int_{\R^d} e^{dt}
	\rho\bigg(e^t x,\frac{e^{\theta t}-1}{\theta}\bigg) e^{-ix\cdot\xi}\dd{x} \\
	&= \frac{1}{M}(2\pi)^{-d/2}\int_{\R^d}\rho\bigg(y,\frac{e^{\theta t}-1}{\theta}\bigg)
	e^{-iy\cdot(\exp(-t)\xi)}\dd{y} \\
	&= \frac{1}{M}\widehat\rho\bigg(e^{-t}\xi,\frac{e^{\theta t}-1}{\theta}\bigg).
\end{align*}
With this expression and the substitution $\eta=e^{-t}\xi$, we find that
\begin{align}
  (-\Delta)^{\theta/2}u &= (2\pi)^{-d/2}\int_{\R^d}|\xi|^\theta\widehat u(\xi,t)
	e^{ix\cdot\xi}\dd{\xi} \nonumber \\
	&= \frac{1}{M}(2\pi)^{-d/2}\int_{\R^d}|\xi|^\theta\widehat\rho
	\bigg(e^{-t}\xi,\frac{e^{\theta t}-1}{\theta}\bigg)e^{ix\cdot\xi}\dd{\xi} \nonumber \\
	&= \frac{1}{M}(2\pi)^{-d/2}e^{(d+\theta)t}\int_{\R^d}|\eta|^\theta
	\widehat\rho\bigg(\eta,\frac{e^{\theta t}-1}{\theta}\bigg)e^{i(\exp(t)x)\cdot\eta}
	\dd{\eta} \nonumber \\
	&= \frac{1}{M}e^{(d+\theta)t}(-\Delta)^{\theta/2}\rho
	\bigg(e^t x,\frac{e^{\theta t}-1}{\theta}\bigg). \label{2.aux2}
\end{align}
Adding \eqref{2.aux1}-\eqref{2.aux2} and inserting \eqref{1.eq}
leads to
\begin{align}
  \pa_t u &+ (-\Delta)^{\theta/2}u - \diver(u\na V) 
	= \frac{1}{M}e^{(d+\theta)t}\diver(\rho\na\psi) \nonumber \\
	&= \frac{1}{M}e^{(d+\theta)t}(\na\rho\cdot\na\psi)
	\bigg(e^t x,\frac{e^{\theta t}-1}{\theta}\bigg)
	- \frac{1}{M}e^{(d+\theta)t}\rho^2\bigg(e^t x,\frac{e^{\theta t}-1}{\theta}\bigg).
	\label{2.aux3}
\end{align}
It remains to express the right-hand side in terms of $u$.

By the representation formula for solutions of the Poisson equation $-\Delta \phi=u$
and the substitution $z=e^t y$, it follows that
\begin{align}
  \na\phi(x) &= c_d\int_{\R^d}\frac{x-y}{|x-y|^d}u(y,t)\dd{y}
	= \frac{c_d}{M}e^{dt}\int_{\R^d}\frac{x-y}{|x-y|^d}
	\rho\bigg(e^t y,\frac{e^{\theta t}-1}{\theta}\bigg)\dd{y} \label{2.naphi} \\
  &= \frac{c_d}{M}e^{(d-1)t}\int_{\R^d}\frac{e^tx-z}{|e^t x-z|^d}
	\rho\bigg(z,\frac{e^{\theta t}-1}{\theta}\bigg)dz
	= \frac{1}{M}e^{(d-1)t}\na\psi\bigg(e^t x,\frac{e^{\theta t}-1}{\theta}\bigg),
	\nonumber
\end{align}
where $c_d=\Gamma(d/2)/(2\pi^{d/2})$. Hence, since
$$
  \na u(x,t) = \frac{1}{M}e^{(d+1)t}
	\na\rho\bigg(e^t x,\frac{e^{\theta t}-1}{\theta}\bigg),
$$
the first term on the right-hand side of \eqref{2.aux3} becomes
$$
  \frac{1}{M}e^{(d+\theta)t}(\na\rho\cdot\na\psi)
	\bigg(e^t x,\frac{e^{\theta t}-1}{\theta}\bigg)
	= Me^{-(d-\theta)t}(\na u\cdot\na\phi)(x,t).
$$
Taking the square of the definition of $u$, the second term on the right-hand side
of \eqref{2.aux3} can be written as
$$
  \frac{1}{M}e^{(d+\theta)t}\rho^2\bigg(e^t x,\frac{e^{\theta t}-1}{\theta}\bigg)
	= Me^{-(d-\theta)t}u^2(x,t).
$$
Inserting the previous two expressions in \eqref{2.aux3} and using
$\na u\cdot\na\phi-u^2=\diver(u\na\phi)$, we finish the proof.
\end{proof}

\begin{lemma}\label{lem.bound}
The solution $u$ to \eqref{2.equ} satisfies $u\in 
L^\infty(0,\infty;L^1(\R^d)\cap L^\infty(\R^d))$.
\end{lemma}

\begin{proof}
We use the estimate in \eqref{1.p2} and the definition of $u$ to find that
\begin{align*}
 \|u(t)\|_{L^p(\R^d)}^p
  &= \frac{e^{dpt}}{M^p}\int_{\R^d}
	\rho\bigg(e^t x, \frac{e^{\theta t}-1}{\theta}\bigg)^p \dd{x}
   = \frac{e^{dt(p-1)}}{M^p}\int_{\R^d}
	\rho\bigg(y,\frac{e^{\theta t}-1}{\theta}\bigg)^p \dd{y} \\
  &\le \frac{C}{M^p} e^{dt(p-1)} 
	\bigg(1 + \frac{e^{\theta t}-1}{\theta}\bigg)^{-d(p-1)/\theta}
  \le \frac{C}{M^p} 
	\bigg(\frac{\theta e^{\theta t}}{\theta + e^{\theta t}-1}\bigg)^{d(p-1)/\theta}
  \le \frac{C}{M^p}.
\end{align*}
Consequently, $u(t)$ is bounded in any $L^p$ norm uniformly in time and we conclude
by passing to the limit $p\to\infty$.
\end{proof}

{\em Step 2: Time decay of $u(t)$.} We show that $u$ converges exponentially fast to
the fundamental solution $G_\theta(1/\theta)$ (see Section \ref{sec.frac}). 
First, we relate the difference of $\rho$ to the fundamental solution $G_\theta$ 
and the difference of $u$ to $G_\theta$.

\begin{lemma}\label{lem.rhou}
Let $s=(e^{\theta t}-1)/\theta$. Then it holds that
$$
  \|\rho(s)-MG_\theta(s+1/\theta)\|_{L^1(\R^d)}
	= M\|u(t) - G_\theta(1/\theta)\|_{L^1(\R^d)}.
$$
\end{lemma}

\begin{proof}
By the definition of $u(x,t)$ and the substitution $y=e^t x$, we compute
\begin{align*}
  M\|u(t) - G_\theta(1/\theta)\|_{L^1(\R^d)}
	&= \int_{\R^d}\bigg|e^{dt}\rho\bigg(e^t x,\frac{e^{\theta t}-1}{\theta}\bigg)
	- MG_\theta(x,1/\theta)\bigg|\dd{x} \\
	&= \int_{\R^d}\bigg|\rho\bigg(y,\frac{e^{\theta t}-1}{\theta}\bigg)
	- Me^{-dt}G_\theta(e^{-t}y,1/\theta)\bigg|\dd{y}.
\end{align*}
We set $s=(e^{\theta t}-1)/\theta$ and use the self-similar form 
\eqref{a.self} with $\lambda=e^{-t}=(1+s\theta)^{-1/\theta}$:
\begin{align*}
  M\|u(t) - G_\theta(1/\theta)\|_{L^1(\R^d)}
	&= \int_{\R^d}\big|\rho(x,s) - M(1+s\theta)^{-d/\theta}
	G_\theta\big((1+s\theta)^{-1/\theta}x,1/\theta\big)\big|\dd{x} \\
	&= \int_{\R^d}\big|\rho(x,s) - MG_\theta\big(x,(1+s\theta)/\theta\big)\big|\dd{x} \\
  &= \|\rho(s)-MG_\theta(s+1/\theta)\|_{L^1(\R^d)},
\end{align*}
which concludes the proof.
\end{proof}

The previous lemma shows that it is sufficient to estimate $u(t)-G_\theta(1/\theta)$.

\begin{lemma}\label{lem.decayu}
Let the assumptions of Theorem \ref{thm.main} hold. Then 
the solution $u$ to \eqref{2.equ} satisfies
$$
  \|u(t)-G_\theta(1/\theta)\|_{L^1(\R^d)} \le Ce^{-\theta t/2}, \quad t>0,
$$
where $C>0$ is some constant.
\end{lemma}

\begin{proof}
Let $1<p<2$ and set $v(t)=u(t)/G_\theta(1/\theta)$ and
\begin{equation}\label{2.ent}
  E_p[v(t)] = \int_{\R^d}v(x,t)^p G_\theta(x,1/\theta)\dd{x}
	- \bigg(\int_{\R^d}v(x,t)G_\theta(x,1/\theta)\dd{x}\bigg)^p.
\end{equation}
We prove in Step 4 below that this functional is well defined.
Note that $E_p[v]=\textrm{Ent}_{u_\infty}^\Phi(v)$, where $\Phi(s)=s^p$ and 
$\textrm{Ent}_{u_\infty}^\Phi$ is the entropy defined in Section \ref{sec.levy}.
We differentiate $E_p[v]$ with respect to time. The derivative of the
second integral in $E_p[v]$
vanishes since $\int_{\R^d}vG_\theta(1/\theta)\dd{x}=\int_{\R^d}u\dd{x}=1$. 
Therefore, by \eqref{2.equ},
\begin{align*}
  \frac{\dd{E_p}}{\dd{t}}[v(t)] &= p\int_{\R^d}v(x,t)^{p-1}\pa_t u(x,t)\dd{x} \\
	&= p\int_{\R^d}v^{p-1}\big(-(-\Delta)^{\theta/2}u + \diver(u\na V)\big)\dd{x}
	+ pMe^{-(d-\theta)t}\int_{\R^d}v^{p-1}\diver(u\na\phi)\dd{x}.
\end{align*}
Using the L\'evy operator $\mathcal{I}[u]=\diver(x\na V)-(-\Delta)^{\theta/2}$ 
with $V(x)=\frac12|x|^2$,
the calculations in the proof of \cite[Proposition 1]{GeIm08} show that
\begin{align*}
  p\int_{\R^d} & v^{p-1} \big(-(-\Delta)^{\theta/2}u + \diver(u\na V)\big)\dd{x} \\
  &= -\int_{\R^d}\int_{\R^d}D_p(v(x,t),v(x+y,t))\nu(\dd{y})G_\theta(x,1/\theta)\dd{x},
\end{align*}
where $D_p(a,b)=a^p+b^p-pb^{p-1}(a-b)$ is the Bregman distance
(see Section \ref{sec.levy}) and $\nu(\dd{y})=\dd{y}/|y|^{d+\theta}$. 
Moreover, the modified logarithmic Sobolev inequality \eqref{a.logsob} gives
$$
  \int_{\R^d}\int_{\R^d}D_p(v(x,t),v(x+y,t))\nu(\dd{y})G_\theta(x,1/\theta)\dd{x}
	\ge \theta E_p[v],
$$
since $u_\infty(x)=G_\theta(x,1/\theta)$ is the density of an infinitely 
divisible probability measure with $\sigma_\infty=0$ and L\'evy measure 
$\nu_\infty(\dd{y})=\dd{y}/\theta |y|^{d+\theta}$.
Putting these estimates together leads to
\begin{equation}\label{2.dEdt}
  \frac{\dd{E_p}}{\dd{t}}[v]  \le -\theta E_p[v]
	+ pMe^{-(d-\theta)t}\int_{\R^d}v^{p-1}\diver(vG_\theta(1/\theta)\na\phi)\dd{x},
\end{equation}
and it remains to estimate the last integral. 
To this end, we differentiate and integrate by parts:
\begin{align*}
  p\int_{\R^d}v^{p-1}\diver(vG_\theta(1/\theta)\na\phi)\dd{x}
	&= \int_{\R^d}\big(pv^p\diver(G_\theta(1/\theta)\na\phi) 
	+ \na(v^p)\cdot\big(G_\theta(1/\theta)\na\phi\big)\big)\dd{x} \\
	&= (p-1)\int_{\R^d}v^p\diver(G_\theta(1/\theta)\na\phi)\dd{x}.
\end{align*}
Differentiating $G_\theta(1/\theta)\na\phi$ and using $-\Delta\phi=u$, we obtain
\begin{align*}
  \bigg|\int_{\R^d}&v^{p}\diver(G_\theta(1/\theta)\na\phi)\dd{x}\bigg| \\
	&\le \int_{\R^d}\big|v^p\na\log G_\theta(1/\theta)\cdot\na\phi
	\big|G_\theta(1/\theta)\dd{x}
	+ \int_{\R^d}v^p uG_\theta(1/\theta)\dd{x} \\
	&\le \big(\|\na\log G_\theta(1/\theta)\|_{L^\infty(\R^d)}
	\|\na\phi\|_{L^\infty(\R^d)}
	+ \|u\|_{L^\infty(\R^d)}\big)\int_{R^d}v^p G_\theta(1/\theta)\dd{x}.
\end{align*}
We claim that the $L^\infty$ norms are bounded uniformly in time. 
Indeed, by Lemma \ref{lem.bound}, the $L^\infty$ norm of $u(t)$ is uniformly bounded.
By \eqref{a.grad}, we find that
$$
  |\na\log G_\theta(1/\theta)| = \frac{|\na G_\theta(1/\theta)|}{G_\theta(1/\theta)}
	\le C\min\{\theta^{1/\theta},|x|^{-1}\} \le C,
$$
which shows the bound for $\|\na\log G_\theta(1/\theta)\|_{L^\infty(\R^d)}$.
Finally, we deduce from Poisson's representation formula that
\begin{align*}
  |\na\phi(x,t)| &\le C\int_{\R^d}\frac{|u(y,t)|}{|x-y|^{d-1}}\dd{y} \\
	&= C\int_{\{|x-y|<1\}}\frac{|u(y,t)|}{|x-y|^{d-1}}\dd{y}
	+ C\int_{\{|x-y|\ge 1\}}\frac{|u(y,t)|}{|x-y|^{d-1}}\dd{y} \\
	&\le C\|u(t)\|_{L^\infty(\R^d)}\int_{\{|x-y|<1\}}\frac{\dd{y}}{|x-y|^{d-1}}
	+ C\int_{\R^d}|u(y,t)|\dd{y} \\
	&\le C\big(\|u(t)\|_{L^\infty(\R^d)}
	+ \|u(t)\|_{L^1(\R^d)}\big),
\end{align*}
and we have already seen that the right-hand side is bounded, 
due to Lemma~\ref{lem.bound}. Hence, \eqref{2.dEdt} becomes
$$
  \frac{\dd{E_p}}{\dd{t}}[v] \le -\theta E_p[v] 
	+ (p-1)CMe^{-(d-\theta)t}\int_{\R^d}v^p G_\theta(1/\theta)\dd{x},
$$
where $C:=\sup_{0<t<\infty}(\|\na\log G_\theta(1/\theta)\|_{L^\infty(\R^d)}
\|\na\phi\|_{L^\infty(\R^d)} + \|u\|_{L^\infty(\R^d)})$.
By definition of the entropy $E_p[v]$ and the mass conservation $\int_{\R^d}u\dd{x}=1$, 
the integral on the right-hand side equals
$$
  \int_{R^d}v^p G_\theta(1/\theta)\dd{x} 
	= E_p[v] + \bigg(\int_{\R^d}vG_\theta \dd{x}\bigg)^p
	= E_p[v] + 1,
$$
and we end up with
$$
  \frac{\dd{E_p}}{\dd{t}}[v] \le \big(-\theta + \sigma(t)\big) E_p[v] + \sigma(t),
	\quad\mbox{where }\sigma(t)=(p-1)CMe^{-(d-\theta)t}.
$$
We apply the Gronwall inequality to infer that
$$
  E_p[v(t)] \le e^{-\theta t + S(t)}\bigg(E_p[v(0)] + \int_0^t e^{\theta s-S(s)}
	\sigma(s)\dd{s}\bigg),
$$
where $S(t)=\int_0^t\sigma(s)\dd{s}$ is bounded uniformly in $t\in(0,\infty)$
since $d-\theta>0$. It follows from
$e^{\theta s}\sigma(s)=(p-1)CMe^{(2\theta-d)s}$ and $2\theta-d<0$ that, 
for some constants $C>0$,
$$
  E_p[v(t)] \le Ce^{-\theta t} \bigg(E_p[v(0)] + \int_0^t e^{(2\theta-d)s}\dd{s}\bigg)
	\le Ce^{-\theta t}.
$$

Finally, we use the Csisz\'ar-Kullback inequality (see \cite{AMTU01} or 
\cite[Theorem A.3]{Jue16}), applied to the function $z\mapsto z^p-1$ for $1<p<2$,
\begin{align*}
  \|u(t)-G_\theta(1/\theta)\|_{L^1(\R^d)}^2
	&\le \frac{2}{p(p-1)}\int_{\R^d}
	\bigg\{\bigg(\frac{u}{G_\theta(1/\theta)}\bigg)^p-1\bigg\}G_\theta(1/\theta)\dd{x} \\
	&= \frac{2}{p(p-1)}\int_{\R^d}(v^p-1)G_\theta(1/\theta)\dd{x} \\
	&= \frac{2}{p(p-1)}\bigg(\int_{\R^d}v^pG_\theta(1/\theta)\dd{x} - 1\bigg) \\
	&= \frac{2}{p(p-1)}E_p[v(t)] 
	\le Ce^{-\theta t},
\end{align*}
where we used the fact that $G_\theta(1/\theta)$ is normalized. 
\end{proof}

{\em Step 3: Time decay of $\rho(t)$.} Lemmas \ref{lem.rhou} and \ref{lem.decayu}
show that
\begin{align*}
  \|\rho(s)-MG_\theta(s+1/\theta)\|_{L^1(\R^d)}
	&= M\|u(t)-G_\theta(1/\theta)\|_{L^1(\R^d)} \\
	&\le CMe^{-\theta t/2}
	= CM(1+\theta s)^{-1/2},
\end{align*}
since $s=(e^{\theta t}-1)/\theta$ is equivalent to $e^{\theta t}=1+\theta s$.
We infer from Lemma \ref{lem.heat} that
\begin{align*}
  \|\rho(s)-MG_\theta(s)\|_{L^1(\R^d)}
	&\le \|\rho(s)-MG_\theta(s+1/\theta)\|_{L^1(\R^d)} \\
	&\phantom{xx}{}+ M\|G_\theta(s+1/\theta)-G_\theta(s)\|_{L^1(\R^d)} \\
	&\le CM (1+\theta s)^{-1/2}.
	\nonumber
\end{align*}

{\em Step 4: Well-posedness of the entropy \eqref{2.ent}.}
It remains to verify that \eqref{2.ent} is well defined. We reformulate $E_p[v]$
in terms of $\rho$. With the substitutions $y=e^{t}x$ and $s=(e^{\theta t}-1)/\theta$,
we find from definition \eqref{2.defu} that
\begin{align*}
  E_p[v] &= \frac{e^{dpt}}{M^p}\int_{\R^d}\rho\bigg(e^t x,\frac{e^{\theta t}-1}{\theta}
	\bigg)^p G_\theta(x,1/\theta)^{1-p}\dd{x} 
	- \frac{e^{dpt}}{M^p}\bigg(\int_{\R^d}\rho\bigg(e^t x,\frac{e^{\theta t}-1}{\theta}
	\bigg)\dd{x}\bigg)^p \\
	&= \frac{e^{d(p-1)t}}{M^p}\int_{\R^d}\rho(y,s)^p G_\theta(e^{-t}y,1/\theta)^{1-p}\dd{y}
	- \frac{1}{M^p}\bigg(\int_{\R^d}\rho(y,s)\dd{y}\bigg)^p.
\end{align*}
Using the self-similar form \eqref{a.self} with $\lambda=e^{-t}$ yields
$$
  E_p[v] = \frac{1}{M^p}\int_{\R^d}\bigg(\frac{\rho(y,s)}{G_\theta(y,s+1/\theta)}
	\bigg)^p G_\theta(y,s+1/\theta)\dd{y}
	- \frac{1}{M^p}\bigg(\int_{\R^d}\rho(y,s)\dd{y}\bigg)^p.
$$
Therefore, it is sufficient to show that the integral
\begin{equation}\label{2.int}
  I(t) := \int_{\R^d}\bigg(\frac{\rho(x,t)}{G_\theta(x,t+1/\theta)}\bigg)^p 
	G_\theta(x,t+1/\theta)\dd{x}
\end{equation}
converges for any fixed $t>0$ and some $1<p<2$. 
The trivial estimate (taking into account \eqref{a.asymp})
$$
  I(t) \le \|\rho(t)\|^p_{L^\infty(\R^d)}\int_{\R^d}G(x,t+1/\theta)^{1-p}\dd{x}
	\le Ct^{-d/\theta}\int_{\R^d}(1+|x|)^{(d+\theta)(p-1)}\dd{x}
$$
cannot be used since $p-1>0$, so we have to derive a finer estimate for 
$\rho(t)$ in some $L^q(\R^d)$. This is done in the following lemma.

\begin{lemma}\label{lem.estrho}
Let $d\ge 2$ and either $1<\theta\le 2$, $q>d/(d-\theta)$ or
$0<\theta\le 1$, $q>d/\theta$. Furthermore, let $|x|^d\rho_0\in L^q(\R^d)$, 
and the solution to \eqref{1.eq} satisfies \eqref{1.p1}-\eqref{1.p2}. Then
$$
  \big\||x|^d\rho(t)\big\|_{L^q(\R^d)} \le C(1+t)^{d/(\theta q)}, \quad t>0.
$$
\end{lemma}

We postpone the (lengthy) proof to Appendix \ref{sec.estrho}.

We claim that
Lemma \ref{lem.estrho} and estimate \eqref{a.asymp} imply that \eqref{2.int} converges.
To see this, we first observe that \eqref{2.int} converges for $p=1$. 
For $p>1$, we fix $t>0$ and split \eqref{2.int} into two parts, $I(t)=I_1+I_2$, where
\begin{align*}
  I_1 &= \int_{B_R(0)}\rho(x,t)^p G_\theta(x,t+1/\theta)^{1-p}\dd{x}, \\
	I_2 &= \int_{\R^d\backslash B_R(0)}\rho(x,t)^p G_\theta(x,t+1/\theta)^{1-p}\dd{x},
\end{align*}
and $B_R(0)$ is the ball of radius $R>0$ centered at the origin. 
Let $K=R(t+1/\theta)^{-1/\theta}$. Consider
$|x|<R=K(t+1/\theta)^{1/\theta}$. 
Then we can apply estimates \eqref{a.asymp} and \eqref{1.p2}:
$$
  I_1 \le C(t+1/\theta)^{\frac{d}{\theta}(p-1)}\int_{\R^d}\rho(x,t)^p \dd{x}
	\le C(t+1/\theta)^{\frac{d}{\theta}(p-1)}t^{-\frac{d}{\theta}(p-1)}\le C.
$$
To estimate $I_2$, we wish to apply Lemma \ref{lem.estrho} 
for large $q>d/\theta$. If $0<\theta<1$, the inequality $q>d/\theta$ shows that 
the assumption of Lemma \ref{lem.estrho} is satisfied. 
If $1<\theta\le 2$, it follows that our assumption 
$\theta<d/2$ (in Theorem~\ref{thm.main})
is equivalent to $d/\theta>d/(d-\theta)$ such that $q>d/(d-\theta)$,
and Lemma \ref{lem.estrho} is applicable also in this case.
Consider $1<p<2$ and $\min\{p,d/\theta\}<q$. 
We employ H\"older's inequality for $r=q/(q-p)$ to find that
\begin{align} 
  I_2 &= \int_{\R^d\backslash B_R(0)} |x|^{dp} \rho(x,t)^p 
	G_\theta(x,t+1/\theta)^{1-p} |x|^{-dp} \dd{x} \nonumber \\
  &\leq \big\||x|^{d} \rho(x,t) \big\|_{L^q(\R^d)}^p 
  \bigg( \int_{\R^d\backslash B_R(0)} G_\theta(x,t+1/\theta)^{(1-p)r} |x|^{-dpr} 
	\dd{x}\bigg)^{1/r}. \label{est.I2}
\end{align}
Using the pointwise estimates \eqref{a.asymp}, the second factor can be estimated as
\begin{align*}
  \int_{\R^d\backslash B_R(0)} G_\theta(x,t+1/\theta)^{(1-p)r} |x|^{-dpr} \dd{x}
   &\leq \int_{\R^d\backslash B_R(0)} \big(C (t+1/\theta) |x|^{-d-\theta} 
	\big)^{(1-p)r} |x|^{-dpr} \dd{x} \\
   &\leq C (t+1/\theta)^{(1-p)r} \int_{\R^d\backslash B_R(0)} 
	|x|^{(d+\theta) (p-1)r -dpr} \dd{x}\\
   &\leq C (t+1/\theta)^{(1-p)r} \int_R^\infty \eta^{(d+\theta) (p-1)r -dpr +d-1} 
	\dd{\eta}.
\end{align*}
The last integral exists if and only if 
$(d+\theta) (p-1)r -dpr +d-1 < -1$ or equivalently $q<dp/((p-1)\theta)$.
Since $dp/((p-1)\theta)\to\infty$ as $p\to 1$, for any large $q>d/\theta$, 
we can choose a sufficiently small $p>1$ such that 
$d/\theta<q< dp/((p-1)\theta)$.
Collecting the estimates starting from \eqref{est.I2} with 
$R=K(t+1/\theta)^{1/\theta}$ and using Lemma \ref{lem.estrho}, we deduce that
\begin{align*}
 I_2
  &\leq C(1+t)^{dp/(\theta q)} (t+1/\theta)^{1-p} |R|^{(d+\theta) (p-1) -dp +d/r} \\
  &\leq C(1+t)^{dp/(\theta q)} (t+1/\theta)^{1-p} 
	(t+1/\theta)^{((d+\theta) (p-1) -dp +d/r)/\theta} \leq C
\end{align*}
is uniformly bounded in time.
Thus, the integral $I(t)$ in~\eqref{2.int} and consequently $E_p[v(t)]$ are well-defined 
for all $t\geq 0$. This finishes the proof of Theorem \ref{thm.main}.

\section{Application to other equations}\label{sec.other}

The entropy method can be applied to other semilinear equations, and in this section,
we give two illustrative examples.

\subsection{Quasi-geostrophic equations}

The two-dimensional quasi-geostrophic equations (for sufficiently smooth solutions)
read as
\begin{equation}\label{4.geo}
  \pa_t\rho + (-\Delta)^{\theta/2}\rho = \diver(\rho\na^\perp\psi),
	\quad (-\Delta)^{1/2}\psi = -\rho\quad\mbox{in }\R^2,\ t>0,
\end{equation}
with the initial condition $\rho(0)=\rho_0\ge 0$ in $\R^2$, $0<\theta\le 2$, 
and $\na^\perp=(-\pa_2,\pa_1)^T$. 
This model approximates the atmospheric and oceanic fluid flow in a 
certain physical regime. The variables $\rho$ and $\psi$ refer to the temperature 
and the stream function, respectively. 
For the geophysical background, we refer to \cite{CMT94}. 
The existence of global solutions in a Sobolev space setting was shown in 
\cite{Ju04}, for instance.
The time decay in the $L^\infty(\R^2)$ norm was investigated in \cite{CoCo04}.
Estimate \eqref{1.p2} can be obtained from standard estimates 
 since the nonlinear part vanishes after multiplication by $\rho^{p-1}$ ($p>1$) 
 and integration over $\R^2$:
$$
  \int_{\R^2}\rho^{p-1}\diver(\rho\na^\perp\psi)\dd{x}
	= (p-1)\int_{\R^2}\rho^{p-1}\na\rho\cdot\na^\perp\psi \dd{x} = 0.
$$
For the analysis of the asymptotic profile,
we use the transformation \eqref{2.defu} 
and $\phi(x,t) := (e^{t}/M) \psi(e^t x,(e^{\theta t}-1)/\theta)$, yielding
\begin{align*}
  & \pa_t u + (-\Delta)^{\theta/2}u = \diver(u\na V) + Me^{-(3-\theta)t}
	\diver(u\na^\perp\phi), \\
	& (-\Delta)^{1/2}\phi = -u \quad\mbox{in }\R^2,\ t>0, \quad u(0)=\rho_0/M,
\end{align*}
where $V(x)=\frac12|x|^2$.
Defining the entropy as in \eqref{2.ent}, a computation similar to the one 
in the proof of Lemma \ref{lem.decayu} gives for $v=u/G_\theta(1/\theta)$:
\begin{align*}
  \frac{\dd{E_p}}{\dd{t}}[v] \le -\theta E_p[v]
	+ pMe^{-(3-\theta)t}\int_{\R^2}v^{p-1} G_\theta(1/\theta)^{1-1/p}u
	\na\big(G_\theta(1/\theta)^{1-1/p}\big)\cdot\na^\perp\phi \dd{x}.
\end{align*}
Estimate \eqref{a.grad} shows that $\na(G_\theta(1/\theta)^{1-1/p})\in L^\infty(\R^2)$.
Moreover,
$$
  \sup_{t>0}\|\na^\perp\phi\|_{L^p(\R^2)} 
	= \sup_{t>0}\|\na^\perp(-\Delta)^{-1/2}u\|_{L^p (\mathbb{R}^2)}
	\le C\sup_{t>0}\|u\|_{L^p(\R^d)}
$$
is finite. 
The regularity of $\rho$ implies that $u\in L^\infty(0,\infty;L^\infty(\R^2))$.
Therefore, using H\"older's and Young's inequalities, we find that
\begin{align*}
  \bigg|\int_{\R^2}&v^{p-1} G_\theta(1/\theta)^{1-1/p}u
	\na\big(G_\theta(1/\theta)^{1-1/p}\big)\cdot\na^\perp\phi \dd{x}\bigg| \\
  &\le \|\na(G_\theta(1/\theta)^{1-1/p})\|_{L^\infty(\R^2)}\|u\|_{L^\infty(\R^2)}
	\|\na^\perp\phi\|_{L^p(\R^2)}\bigg(\int_{\R^2}v^p G_\theta(1/\theta)\dd{x}\bigg)^{1-1/p}
	\\
	&\le C\big(E_p[v]+1\big).
\end{align*}
Thus, the same procedure as in Section \ref{sec.proof} shows that
the solution to \eqref{4.geo} satisfies \eqref{1.decay}.


\subsection{Generalized Burgers equation}

The generalized Burgers equation in one space dimension
\begin{equation}\label{4.burger}
  \pa_t\rho + (-\pa_x^2)^{\theta/2}\rho + \rho\pa_x\rho = 0\quad\mbox{in }\R,\ t>0,
	\quad \rho(0)=\rho_0,
\end{equation}
with $\theta=2$ is a basic model for one-dimensional fluid flows with
interacting nonlinear and dissipative phenomena,
and it describes for $0<\theta<2$ the overdriven detonation in gases \cite{Cla02}
or the anomalous diffusion in semiconductor growth \cite{Woy01}.
The value $\theta=1$ is a
threshold for the occurrence of singularities \cite{DDL09}, and the existence of 
global solutions holds for $1<\theta<2$. 
Global solutions in Besov spaces were also shown to exist \cite{Iwa15},
and in the same paper, the self-similar asymptotics have been analyzed.
The function $u$, defined by \eqref{2.defu}, solves
$$
  \pa_t u + (-\pa_x^2)^{\theta/2}u = \pa_x(u\pa_x V) - Me^{-(2-\theta)t}u\pa_x u,
$$
and the entropy method yields, similar as in the previous subsection,
\begin{align*}
  \frac{\dd{E_p}}{\dd{t}}[v] &\le -\theta E_p[v]
	+ pMe^{-(2-\theta)t}\int_{\R}\pa_x u\, v^p G_\theta(1/\theta)\dd{x} \\
	&\le -\theta E_p[v] + pMe^{-(2-\theta)t}\|\pa_x u\|_{L^\infty(\R)}(E_p[v]+1).
\end{align*}
Assuming that the solution to \eqref{4.burger} satisfies the bound
$\|\pa_x\rho(t)\|_{L^\infty(\R)}\le Ct^{-2/\theta}$, it follows that
$\pa_x u\in L^\infty(0,\infty;L^\infty(\R))$, and we obtain the decay estimate
\eqref{1.decay}, where a restriction for the range of $\theta$ is expected
from estimating the nonlinear terms. 

Observe that the entropy method does not work in the case $\theta=2$
(see the discussion at the end of the introduction in \cite{Iwa15}).
Moreover, when $0<\theta<1$, there are initial data such that the solution
develops gradient blowup in finite time \cite{DDL09}.


\begin{appendix}

\section{Proof of Lemma \ref{lem.estrho}}\label{sec.estrho}

For later reference, we note that,
by the Poisson representation formula 
and property \eqref{1.p2} with $p=1$ and $p=\infty$,
\begin{align}
  |\na\psi(x,t)| &= |\na(-\Delta)^{-1}\rho(x,t)|
	= c_d\bigg|\int_{\R^d}\frac{x-y}{|x-y|^d}\rho(y,t)\dd{y}\bigg| \nonumber \\
	&\le c_d\int_{\{|x-y|\le (1+t)^{1/\theta}\}}\frac{|\rho(y,t)|}{|x-y|^{d-1}}\dd{y}
	+ c_d\int_{\{|x-y| > (1+t)^{1/\theta}\}}\frac{|\rho(y,t)|}{|x-y|^{d-1}}\dd{y} 
	\nonumber \\
  &\le C(1+t)^{1/\theta}\|\rho(t)\|_{L^\infty(\R^d)}
	+ C(1+t)^{-(d-1)/\theta}\|\rho(t)\|_{L^1(\R^d)} \nonumber \\
	&\le C(1+t)^{-d/\theta+1/\theta}. \label{3.psi}
\end{align}

{\em Step 1: $1<\theta\le 2$.}
The solution to \eqref{1.eq} can be formulated as the mild solution
$$
  \rho(t) = G_\theta(t)*\rho_0 + \int_0^t\na G_\theta(t-s)*(\rho\na(-\Delta)^{-1}\rho)
	(s)\dd{s}.
$$
We multiply this equation by $|x|^d$,
\begin{align*}
  \big||x|^d\rho(t)\big|
	&\le \big||x|^dG_\theta(t)*\rho_0\big| \\
	&\phantom{xx}{}+ C\int_0^t\int_{\R^d}\big(|x-y|^d+|y|^d\big)|\na G_\theta(t-s,x-y)|
	|(\rho\na(-\Delta)^{-1}\rho)(s,y)|\dd{y}\dd{s},
\end{align*}
and take the $L^q(\R^d)$ norm:
\begin{align}
  &\big\||x|^d\rho(t)\big\|_{L^q(\R^d)} 
	\le \big\||x|^d G_\theta(t)\big\|_{L^q(\R^d)}\|\rho_0\|_{L^1(\R^d)}
	+ \|G_\theta (t)\|_{L^1 (\R^d)}\big\| |x|^d \rho_0\big\|_{L^q (\R^d)} \nonumber \\
	&\phantom{xx}{}+ C\int_0^t\big\||x|^d\na G_\theta(t-s)\big\|_{L^q(\R^d)}
	\|\rho(s)\|_{L^1(\R^d)}\big\|(\na(-\Delta)^{-1}\rho)(s)\big\|_{L^\infty(\R^d)} \dd{s} 
	\nonumber \\
	&\phantom{xx}{}+ C\int_0^t\|\na G_\theta(t-s)\|_{L^1(\R^d)}
	\big\||x|^d\rho(s)\big\|_{L^q(\R^d)}
	\big\|(\na(-\Delta)^{-1}\rho)(s)\big\|_{L^\infty(\R^d)} \dd{s} \nonumber \\
	&=: I_1 + I_2 + I_3. \label{3.I123}
\end{align}
We estimate term by term.

Estimates \eqref{a.asymp} show that $\||x|^dG_\theta(t)\|_{L^\infty(\R^d)}\le C$
and $\||x|^dG_\theta(t)\|_{L^q(\R^d)}\le C(1+t)^{d/(q\theta)}$ for $q>d/\theta$, 
and $\| G_\theta (t) \|_{L^1 (\R^d)} = 1$.
We deduce that
$$
  I_1 \le C(1+t)^{d/(q\theta)}.
$$
For the estimate of $I_2$, we first infer from 
the pointwise estimates \eqref{a.grad} that
$$
  \big\||x|^d\na G_\theta(t)\big\|_{L^q(\R^d)}
	\le C \big\|\min\{t^{-1/\theta}, |x|^{-1}\}\big\|_{L^\infty(\R^d)}
	\big\||x|^d G_\theta(t)\big\|_{L^q(\R^d)}
  \le Ct^{-\frac{1}{\theta}+\frac{d}{q\theta}}.
$$
Therefore, taking into account \eqref{3.psi} and the inequality $\theta\le d$,
$$
  I_2 \le C\int_0^t(t-s)^{-1/\theta+d/(q\theta)}(1+s)^{-d/\theta+1/\theta}\dd{s}
	\le C(1+t)^{d/(q\theta) - d/\theta + 1} \le C(1+t)^{d/(q\theta)}.
$$
Note that $\theta>1$ implies that $-1/\theta+d/(q\theta)>-1$ and so, 
$(t-s)^{-1/\theta+d/(q\theta)}$ is integrable in $(0,t)$.

For the integral $I_3$, we first observe that,
again by \eqref{a.asymp} and \eqref{a.grad}, 
$\|\na G_\theta(t-s)\|_{L^1(\R^d)} \le C(t-s)^{-1/\theta}$. Hence, with \eqref{3.psi},
\begin{align*}
  I_3 &\le C\int_0^t(t-s)^{-1/\theta}\big\||x|^d\rho(s)\big\|_{L^q(\R^d)}
	(1+s)^{-d/\theta+1/\theta}\dd{s} \\
	&\le C\sup_{0<s<t}\Big((1+s)^{\frac{d}{\theta}(1-\frac{1}{q})-\frac{d}{\theta}}
	\big\||x|^d\rho(s)\big\|_{L^q(\R^d)}\Big)\int_0^t(t-s)^{-1/\theta}
	(1+s)^{-\frac{d}{\theta}(1-\frac{1}{q})+\frac{1}{\theta}}\dd{s} \\
	&\le C\sup_{0<s<t}\Big((1+s)^{\frac{d}{\theta}(1-\frac{1}{q})-\frac{d}{\theta}}
	\big\||x|^d\rho(s)\big\|_{L^q(\R^d)}\Big)\tau(t),
\end{align*}
where $\tau(t)=(1+t)^{-\frac{d}{\theta}(1-\frac{1}{q})+1}$. 
Note that again, $\theta>1$ is needed to ensure that $(t-s)^{-1/\theta}$ 
is integrable in $(0,t)$.
Since $q>d/(d-\theta)$, it follows that $\tau(t)\to 0$ as
$t\to\infty$. 

We conclude from \eqref{3.I123} that
$$
  \big\||x|^d\rho(t)\big\|_{L^q(\R^d)}
	\le C(1+t)^{d/(q\theta)} + C\tau(t)\sup_{0<s<t}\Big((1+s)^{-d/(q\theta)}
	\big\||x|^d\rho(s)\big\|_{L^q(\R^d)}\Big).
$$
This can be written as
\begin{align*}
  \sup_{0<s<t}&\Big((1+s)^{-d/(q\theta)}
	\big\||x|^d\rho(s)\big\|_{L^q(\R^d)}\Big) \\
	&\le C + (1+t)^{-d/(q\theta)}\tau(t)\sup_{0<s<t}\Big((1+s)^{-d/(q\theta)}
	\big\||x|^d\rho(s)\big\|_{L^q(\R^d)}\Big).
\end{align*}
Choosing $t>0$ sufficiently large, we infer that 
$$
  \sup_{0<s<t}\Big((1+s)^{-d/(q\theta)}
	\big\||x|^d\rho(s)\big\|_{L^q(\R^d)}\Big) \le C,
$$
proving the claim.

{\em Step 2: $0<\theta\le 1$.}
We first study the even dimensional case, $d=2m$ for some $m\in\N$. The function
$P(x,t)=|x|^d\rho(x,t)$ solves
\begin{equation}\label{3.P}
  \pa_t P + (-\Delta)^{\theta/2}P - \diver(P\na\psi)
	= [(-\Delta)^{\theta/2},|x|^d]\rho + [|x|^d,\diver](\rho\na\psi),
\end{equation}
where $[A,B]=AB-BA$ is the commutator. The first term on the right-hand side becomes
$$
  [(-\Delta)^{\theta/2},|x|^d]\rho 
	= \F^{-1}\big[[|\xi|^\theta,(-\Delta)^m]\widehat\rho\big].
$$
The highest order derivative of $\widehat\rho$ in the commutator cancels, since
$$
  [|\xi|^\theta,(-\Delta)^m]\widehat\rho
	= |\xi|^\theta(-\Delta)^m\widehat\rho - (-\Delta)^m(|\xi|^\theta\widehat\rho),
$$
and the chain rule gives the sum of derivatives up to order $2m-1=d-1$.
Hence, we infer from the Fourier convolution formula that
\begin{align*}
  \big[(-\Delta)^{\theta/2},|x|^d\big]\rho 
	&= \sum_{|\alpha|+|\beta|=d,\,|\beta|\le d-1}a_{\alpha,\beta}
	\F^{-1}\big[(i\na)^\alpha(|\xi|^\theta)(i\na)^\beta\widehat\rho\big] \\
	&= \sum_{|\alpha|+|\beta|=d,\,|\beta|\le d-1}a_{\alpha,\beta}
	\F^{-1}\big[(i\na)^\alpha(|\xi|^\theta)\big]*(x^\beta\rho),
\end{align*}
where $\alpha$, $\beta\in\N^d_0$ are multi-indices,
$a_{\alpha,\beta}\in\R$ are some constants, and $(i\na)^\beta$ and $x^\beta$
are to be understood as products of the components of the corresponding vectors.
Evaluating the commutator, the second term on the right-hand side of \eqref{3.P} equals
$$
  [|x|^d,\diver](\rho\na\psi) = |x|^d\diver(\rho\na\psi) - \diver(|x|^d\rho\na\psi)
	= -d|x|^{d-2}\rho x\cdot\na\psi.
$$
Hence, multiplying \eqref{3.P} by $P^{q-1}$ with $q>d/\theta>1$, we have
\begin{align}
  \frac{1}{q}&\frac{\dd{}}{\dd{t}}\int_{\R^d}P^q \dd{x} 
	+ \int_{\R^d}P^{q-1}(-\Delta)^{\theta/2}P\dd{x}
	+ \frac{q-1}{q}\int_{\R^d}\na\psi\cdot\na P^q \dd{x} \nonumber \\
	&= \sum_{|\alpha|+|\beta|=d,\,|\beta|\le d-1}a_{\alpha,\beta} \int_{\R^d}P^{q-1}
	\F^{-1}\big[(i\na)^\alpha(|\xi|^\theta)\big]*(x^\beta\rho)\dd{x} \label{3.aux} \\
	&\phantom{xx}{}- d\int_{\R^d} P^{q-1}|x|^{d-2}\rho x\cdot\na\psi \dd{x}. \nonumber
\end{align}
The Stroock-Varopoulos inequality (see, e.g., \cite[Prop.~2.1]{OgYa09},
\cite[Lemma~1]{CoCo03} or \cite[Theorem 2.1]{LiSe96}) 
gives for the second term on the left-hand side:
$$
  \int_{\R^d}P^{q-1}(-\Delta)^{\theta/2}P\dd{x}
	\ge \frac{2}{q}\int_{\R^d}\big|(-\Delta)^{\theta/4}(P^{q/2})\big|^2 \dd{x}.
$$
Furthermore, integrating by parts and using the Poisson equation, the third term on the 
left-hand side of \eqref{3.aux} becomes
$$
  \frac{q-1}{q}\int_{\R^d}\na\psi\cdot\na P^q \dd{x} 
	= \frac{q-1}{q}\int_{\R^d}P^q\rho \dd{x} \ge 0.
$$
Thus, multiplying \eqref{3.aux} by $q(1+t)^{-\gamma q}$ for some 
$0<\gamma< d/(\theta q)$ and integrating over time,
\begin{align}
  (1&+t)^{-\gamma q}\|P(t)\|_{L^q(\R^d)}^q 
	+ \gamma q\int_0^t(1+s)^{-\gamma q-1}\|P(s)\|_{L^q(\R^d)}^q \dd{s} \nonumber \\
	&\phantom{xx}{}+ 2\int_0^t(1+s)^{-\gamma q}
	\|(-\Delta)^{\theta/4}(P(s)^{q/2})\|_{L^2(\R^d)}^2  \dd{s} \nonumber \\
	&\phantom{xx}{}+ (q-1)\int_0^t(1+s)^{-\gamma q}\int_{\R^d}P(s)^q \rho(s)\dd{x}\dd{s} 
	\nonumber \\
	&\le q\sum_{|\alpha|+|\beta|=d,\,|\beta|\le d-1}a_{\alpha,\beta} 
	\int_0^t(1+s)^{-\gamma q}
	\int_{\R^d}P(s)^{q-1}\F^{-1}\big[(i\na)^\alpha(|\xi|^\theta)\big]
	*(x^\beta\rho(s))\dd{x}\dd{s} \nonumber \\
	&\phantom{xx}{}- d\int_0^t(1+s)^{-\gamma q}\int_{\R^d} P(s)^{q-1}|x|^{d-2}
	\rho(s) x\cdot\na\psi(s) \dd{x}\dd{s} + \||x|^d\rho_0\|_{L^q(\R^d)}^q. \label{3.aux2}
\end{align}
The third and fourth terms on the left-hand side will be neglected, and we
need to estimate only the terms on the right-hand side.

Let us consider the first term on the right-hand side.
The H\"older inequality with $\eta=q/(q-1)>1$ and $\eta'=q$ implies that
$$
  \bigg|\int_{\R^d}P^{q-1}\F^{-1}\big[(i\na)^\alpha(|\xi|^\theta)\big]
	*(x^\beta\rho)\dd{x}\dd{s} \bigg|
	\le \|P\|_{L^q(\R^d)}^{q-1}\|\F^{-1}\big[(i\na)^\alpha(|\xi|^\theta)\big]
	*(x^\beta\rho)\|_{L^q(\R^d)}.
$$
Since $(i\na)^\alpha|\xi|^\theta=|\xi|^{\theta-2|\alpha|} P_{|\alpha|} (\xi)$
for some homogeneous polynomial $P_{|\alpha|}$ of $\xi$ with order $|\alpha|$, we see that
$\F^{-1}[(i\na)^\alpha|\xi|^\theta]*(x^\beta\rho) 
=|\na|^{\theta-2|\alpha|} P_{|\alpha|} (i\nabla) (x^\beta\rho)$.
The operator $|\na|^{-|\alpha|} P_{|\alpha|}(i\na)$ is bounded in
$L^q(\R^d)$ since it is a polynomial of Riesz transforms.
Hence, we can apply the Hardy-Littlewood-Sobolev inequality 
\cite[Section V.1.1, Theorem 1]{Ste70} with $1/r = 1/q + (|\alpha|-\theta)/d$ 
to obtain
$$
  \| |\na|^{\theta-2|\alpha|} P_{|\alpha|} (i\nabla) (x^\beta\rho)\|_{L^q(\R^d)}
  \le  C\| |\na|^{\theta - |\alpha|} (x^\beta \rho) \|_{L^q (\mathbb{R}^d)}
   \le C\|x^\beta\rho\|_{L^r(\R^d)}.
$$
We infer that
$$
  \bigg|\int_{\R^d}P^{q-1}\F^{-1}\big[(i\na)^\alpha(|\xi|^\theta)\big]
	*(x^\beta\rho)\dd{x}\dd{s} \bigg|
	\le C\|P\|_{L^q(\R^d)}^{q-1}\|x^\beta\rho\|_{L^r(\R^d)}.
$$

We proceed by applying the H\"older inequality with $\eta=dq/(|\beta|r)>1$
and $\eta'=dq/(dq-|\beta|r)$ to the last norm:
$$
  \big\||x|^{|\beta|}\rho\|_{L^r(\R^d)}^r
	= \int_{\R^d}|x|^{|\beta|r}\rho^{|\beta|r/d}\cdot\rho^{r(1-|\beta|/d)}\dd{x}
	\le \big\||x|^d\rho\big\|_{L^q(\R^d)}^{|\beta|r/d}
	\|\rho\|_{L^\nu(\R^d)}^{r(1-|\beta|/d)},
$$
where $\nu=(d-|\beta|)qr/(dq-|\beta|r)$. Note that $\nu>1$. Indeed, inserting 
$1/r=1/q+(|\alpha|-\theta)/d$ and $|\alpha|+|\beta|=d$ yields
$$
  \nu = \frac{d-|\beta|}{d/r - |\beta|/q}
	= \frac{d-|\beta|}{d/q+|\alpha|-\theta-|\beta|/q}
	= \frac{|\alpha|}{|\alpha|/q+|\alpha|-\theta},
$$
and this is larger than one if and only if $|\alpha|/q<\theta$. This is true
since $|\alpha|/q\le d/q<\theta$ by the choice of $q$.
Then, splitting the integrand,
\begin{align*}
  \bigg|\int_0^t&(1+s)^{-\gamma q}
	\int_{\R^d}P(s)^{q-1}\F^{-1}\big[(i\na)^\alpha(|\mu|^\theta)\big]
	*(x^\beta\rho(s))\dd{x}\dd{s}\bigg| \\
	&\le C\int_0^t(1+s)^{-\gamma q}\|P\|_{L^q(\R^d)}^{|\beta|/d+q-1}
	\|\rho\|_{L^\nu(\R^d)}^{1-|\beta|/d}\dd{s} \\
	&\le C\int_0^t\big((1+s)^{-\gamma q + \mu - 1}\big)^{1/\mu}
	\|\rho\|_{L^\nu(\R^d)}^{1-|\beta|/d}\big((1+s)^{-\gamma q-1}\big)^{1-1/\mu}
	\|P\|_{L^q(\R^d)}^{|\beta|/d+q-1}\dd{s} \\
	&= C\int_0^t\Big((1+s)^{-\gamma q + \mu - 1}
	\|\rho\|_{L^\nu(\R^d)}^{q}\Big)^{1/\mu}
	\Big((1+s)^{-\gamma q-1}\|P\|_{L^q(\R^d)}^{q}\Big)^{(\mu-1)/\mu}\dd{s}.
\end{align*}
Young's inequality with $\eta=\mu>1$ and $\eta'=\mu/(\mu-1)$ and any $\eps>0$ yields
\begin{align}
  \bigg|\int_0^t&(1+s)^{-\gamma q}
	\int_{\R^d}P(s)^{q-1}\F^{-1}\big[(i\na)^\alpha(|\xi|^\theta)\big]
	*(x^\beta\rho(s))\dd{x}\dd{s}\bigg| \nonumber \\
	&\le C_\eps \int_0^t(1+s)^{-\gamma q + \mu - 1}\|\rho\|_{L^\nu(\R^d)}^{q}\dd{s}
	+ \eps\int_0^t(1+s)^{-\gamma q-1}\|P\|_{L^q(\R^d)}^{q}\dd{s}. \label{3.aux3}
\end{align}
For sufficiently small $\eps>0$, the last term is absorbed by the second term
on the left-hand side of \eqref{3.aux2}. 
In view of \eqref{1.p2}, the first term on the right-hand side of \eqref{3.aux3} is 
estimated according to
$$
  \int_0^t(1+s)^{-\gamma q + \mu - 1}\|\rho\|_{L^\nu(\R^d)}^{q}\dd{s}
	\le C(1+t)^{-\gamma q + \mu - q\frac{d}{\theta}(1-\frac{1}{\nu})}.
$$
Using the definitions of $\mu$, $\nu$, and $1/r=1/q+(|\alpha|-\theta)/d$ 
as well the property $|\alpha|+|\beta|=d$, it follows that
\begin{align*}
  \mu - q\frac{d}{\theta}\bigg(1-\frac{1}{\nu}\bigg)
	&= \frac{dq}{d-|\beta|} - \frac{dq}{\theta}
	\bigg(1 - \frac{dq-|\beta|r}{(d-|\beta|)qr}\bigg) \\
	&= \frac{dq}{d-|\beta|}\bigg(1 - \frac{d-|\beta|}{\theta} + \frac{d}{\theta r}
	- \frac{|\beta|}{q\theta}\bigg) \\
	&= \frac{dq}{d-|\beta|}\bigg(1 - \frac{d}{\theta} + \frac{|\beta|}{\theta}
	+ \frac{d}{q\theta} + \frac{|\alpha|-\theta}{\theta} - \frac{|\beta|}{q\theta}\bigg) \\
	&= \frac{dq}{d-|\beta|}\bigg(\frac{d}{q\theta} - \frac{|\beta|}{q\theta}\bigg)
	= \frac{d}{\theta}.
\end{align*}
Therefore,
$$
  \int_0^t(1+s)^{-\gamma q + \mu - 1}\|\rho\|_{L^\nu(\R^d)}^{q}\dd{s}
	\le C(1+t)^{-\gamma q+d/\theta}.
$$

It remains to bound the second term on the right-hand side of \eqref{3.aux2}. 
We apply H\"older's inequality with $\eta=q/(q-1)>1$ and $\eta'=q$ to obtain
\begin{equation}\label{3.aux4}
  \bigg|\int_{\R^d}P^{q-1}|x|^{d-1}\rho|\na\psi|\dd{x}\bigg|
	\le \|P\|_{L^q(\R^d)}^{q-1}\||x|^{d-1}\rho\|_{L^q(\R^d)}\|\na\psi\|_{L^\infty(\R^d)}.
\end{equation}
The second factor is estimated using H\"older's inequality again with
$\eta=d/(d-1)>1$, $\eta'=1/(1-1/\eta)=d$ and using \eqref{1.p2}:
\begin{align*}
  \||x|^{d-1}\rho\|_{L^q(\R^d)}^q 
	&= \int_{\R^d}|x|^{(d-1)q}\rho^{q/\eta}\cdot\rho^{q(1-1/\eta)} \dd{x} 
	\le \bigg(\int_{\R^d}|x|^{dq}\rho^q \dd{x}\bigg)^{1/\eta}
  \bigg(\int_{\R^d}\rho^{q}\dd{x}\bigg)^{1-1/\eta} \\
	&= \||x|^d\rho\|_{L^q(\R^d)}^{q(1-\frac{1}{d})}\|\rho\|_{L^q(\R^d)}^{q/d}
	\le C\|P\|_{L^q(\R^d)}^{q(1-\frac{1}{d})}(1+s)^{-q/\theta+1/\theta}.
\end{align*}
We conclude from \eqref{3.psi} and \eqref{3.aux4} that
\begin{align*}
  \bigg|\int_{\R^d}P^{q-1}|x|^{d-1}\rho|\na\psi|\dd{x}\bigg|
	\le C(1+s)^{-d/\theta+1/(q\theta)}\|P\|_{L^q(\R^d)}^{q-1/d}.
\end{align*}
Thus, using the Young inequality with $\eta=dq$, $\eta'=dq/(dq-1)$
(we only need that $\eta>0$), for any $\eps>0$, the second term on the
right-hand side of \eqref{3.aux2} becomes
\begin{align*}
  \bigg|\int_0^t&(1+s)^{-\gamma q}\int_{\R^d} P(s)^{q-1}|x|^{d-2}
	\rho x\cdot\na\psi(s) \dd{x}\dd{s}\bigg| \\
	&\le C\int_0^t(1+s)^{-\gamma q - d/\theta + 1/(q\theta)}
	\|P\|_{L^q(\R^d)}^{q-1/d}\dd{s} \\
	&= C\int_0^t(1+s)^{1 + 1/(q\theta) - 1/(dq) - d/\theta - \gamma/d}
	\Big((1+s)^{-\gamma q-1}\|P\|_{L^q(\R^d)}^q\Big)^{1-1/(dq)}\dd{s} \\
	&\le C_\eps\int_0^t(1+s)^{dq + d/\theta - 1 - d^2q/\theta - \gamma q}\dd{s} 
	+ \eps\int_0^t(1+s)^{-\gamma q-1}\|P\|_{L^q(\R^d)}^q \dd{s}. \\
	&\le C(1+t)^{-\gamma q+d/\theta} 
	+ \eps\int_0^t(1+s)^{-\gamma q-1}\|P\|_{L^q(\R^d)}^q \dd{s},
\end{align*}
where we used $dq-d^2q/\theta\le 0$ which is equivalent to $\theta\le d$,
and this property has been assumed.
The last expression can be absorbed by the second term on the left-hand side
of \eqref{3.aux2} for sufficiently small $\eps>0$.

Summarizing the estimates, we conclude from \eqref{3.aux2} (for sufficiently
small $\eps>0$) that
$$
  (1+t)^{-\gamma q}\|P(t)\|_{L^q(\R^d)}^q \le C(1+t)^{-\gamma q+d/\theta},
$$
which gives the desired result.

When the dimension is odd, $d=2m+1$ for some $n\in\N_0$, we choose
$P(x,t)=|x|^{d-1}x\rho(x,t)$ and proceed as above. This concludes the proof.


\section{Formal derivation of the fractional drift-diffusion-Poisson system}
\label{sec.deriv}

Fractional diffusion may be derived from linear kinetic transport models
by using the Fourier-Laplace transform \cite{MMM11} or by exploiting the
harmonic extension definition of the fractional diffusion operator \cite{BGM17}.
An alternative is Mellet's moment method \cite{Mel10}, which was used to derive the
first equation in \eqref{1.eq} with fixed force field $E=\na\psi$ 
\cite{AcMe17}. 
In this section, we sketch a formal derivation of the 
drift-diffusion-Poisson system following \cite{AcMe17} to explain
the origin of equations \eqref{1.eq}. 

The starting point is the scaled Boltzmann equation
\begin{equation}\label{a.BE}
  \eps^{\theta-1}\pa_t f_\eps + v\cdot\na_x f_\eps 
	+ \eps^{\theta-2}\na_x\psi_\eps \cdot\na_v
	f_\eps = \eps^{-1}Q(f_\eps), 
\end{equation}
for the distribution function $f_\eps(x,v,t)$, where $x\in\R^d$ is the
spatial variable, $v\in\R^d$ is the velocity, and $t>0$ is the time.
We prescribe the initial condition $f(x,v,0)=f_I(x,v)$ for $(x,v)\in\R^d\times\R^d$.
The electric potential $\psi_\eps(x,t)$ is governed by the Poisson
equation
\begin{equation}\label{a.poi}
  -\Delta_x\psi_\eps = \rho_\eps, \quad \rho_\eps := \int_{\R^d}f_\eps \dd{v},
\end{equation}
where $\rho_\eps$ is the particle density, and $Q$ is the linear Boltzmann operator
$$
  Q(f) = \int_{\R^d}\big(\sigma(v,v')M(v)f(v') - \sigma(v',v)M(v')f(v)\big)\dd{v'},
$$
where $\sigma(v,v')>0$ is the scattering rate. The parameter $\eps>0$ measures
the collision frequency, and $\theta>0$ fixes the relation between
diffusion, due to scattering, and advection, due to the acceleration field.

System \eqref{a.BE}-\eqref{a.poi} models the evolution of 
particles in a dilute gas subject to a self-consistent acceleration field 
$E_\eps=\na_x\psi_\eps$. For instance, $\rho_\eps$ may describe the density of electrons
in a semiconductor, and $\psi_\eps$ is the electric potential. The function $M(v)$
models the thermodynamic equilibrium. In semiconductor theory, it is typically
given by the Maxwellian, and in this case $\theta=2$ \cite{Jue09}. Here, we are
interested in the case $1<\theta<2$. The main assumptions on the equilibrium
$M$ are that it is positive, normalized, even, and heavy-tailed, i.e.
$$
  M(v)\sim \frac{\gamma}{|x|^{d+\theta}} \quad\mbox{as }|x|\to\infty
$$
for some $\gamma>0$ and $1<\theta<2$.

Equations \eqref{1.eq} are derived in the limit $\eps\to 0$. To this end, we
expand $f_\eps = F_\eps + \eps^{\theta/2}r_\eps$, where $r_\eps$ is a remainder term.
In contrast to standard kinetic theory, $F_\eps$ is not 
given by the equilibrium distribution $M(v)$ (up to a factor)
but it is the unique solution to 
\begin{equation}\label{a.Feps}
  \eps^{\theta-1}E_\eps\cdot\na_{v} F_\eps 
	= Q(F_\eps), \quad \int_{\R^d}F_\eps \dd{v} = 1.
\end{equation}
We can interpret this equation in two ways. First, because of $\theta>1$,
\eqref{a.Feps} converges formally to $Q(F_0)=0$ as $\eps\to 0$ 
with $F_0=\lim_{\eps\to 0}F_\eps$, such that \eqref{a.Feps} is an approximation 
of the equilibrium condition $Q(F_0)=0$.
Second, we can write $F_\eps(x,v,t)=F(v,\eps^{\theta-1}E_\eps(x,t))$, where $F$ is the
unique solution to $E_\eps\cdot\na_vF=Q(F)$,
$\int_{\R^d}F\dd{v}=1$, and this transformation eliminates the factor in $\eps$. 
This equation appears in the high-field limit, first studied 
in kinetic theory for semiconductors by Poupaud \cite{Pou92}. 

Inserting the expansion $f_\eps = F_\eps + \eps^{\theta/2}r_\eps$ into 
the Boltzmann equation \eqref{a.BE},
dividing the resulting equation by $\eps^{\theta/2-1}$,
and observing that $F_\eps$ solves \eqref{a.Feps}, we find that
\begin{equation}\label{a.feps}
  \eps^{\theta/2}\pa_t f_\eps
	+ \eps^{1-\theta/2} v\cdot\na_x (F_\eps+\eps^{\theta/2}r_\eps)
	= -\eps^{\theta-1}E_\eps\cdot\na_x r_\eps + Q(r_\eps).
\end{equation}
Formally, the left-hand side converges to zero as $\eps\to 0$ (since $\theta<2$).
Therefore $R_\eps:=-\eps^{\theta-1}E_\eps\cdot\na_x r_\eps + Q(r_\eps)\to 0$.
In fact, it can be even proven that $\eps^{-\theta/2}R_\eps\to 0$; this is contained
in \cite[Prop.~4.1]{AcMe17}.
We assume that $f_\eps\to f$ as $\eps\to 0$. Hence $\rho_\eps\to \rho=\int_{\R^d}f\dd{v}$.
In view of the Poisson equation,
this implies that $E_\eps\to E$, where $E=\na_x\psi$ and $-\Delta_x\psi=\rho$.

Next, following \cite{AcMe17}, we multiply \eqref{a.feps} by some test function $\chi$
and integrate over $S=\R^d\times\R^d\times(0,T)$:
\begin{equation}\label{a.feps2}
  \iiint_S\pa_t f_\eps\chi \dd{x}\dd{v}\dd{t}
	+ \eps^{1-\theta}\iiint_S 
	v\cdot\na_x(F_\eps+\eps^{\theta/2}r_\eps)\chi \dd{x}\dd{v}\dd{t}
	= \eps^{-\theta/2}\iiint_S R_\eps\chi \dd{x}\dd{v}\dd{t}.
\end{equation}
The key idea is to choose $\chi=\chi_\eps$ as the unique solution to 
$$
  \eps v\cdot\na_x\chi_\eps = \nu(v)(\chi_\eps-\phi) \quad\mbox{in }\R^d,
$$
where $\phi(x,t)$ is some test function and $\nu(v)=\int_{\R^d}\sigma(v',v)M(v')\dd{v'}$.
Formally, $\chi_\eps\to\delta_v\phi$ as $\eps\to 0$, where $\delta_v$ is the Dirac
delta distribution in the velocity space. Therefore, for the first term in
\eqref{a.feps2},
$$
  \iiint_S\pa_t f_\eps\chi_\eps \dd{x}\dd{v}\dd{t} 
	\to \int_0^T\int_{\R^d}\pa_t\rho\phi \dd{x}\dd{t}.
$$
The right-hand side of \eqref{a.feps2} converges to zero. It remains to treat the
second term in \eqref{a.feps2}. Integrating by parts and using the equation
for $\chi_\eps$, we obtain
\begin{align*}
  \eps^{1-\theta}&\iiint_S 
	v\cdot\na_x(F_\eps+\eps^{\theta/2}r_\eps)\chi_\eps \dd{x}\dd{v}\dd{t} \\
	&= -\eps^{1-\theta}\iiint_S v\cdot\na_x\chi_\eps F_\eps \dd{x}\dd{v}\dd{t} 
	- \eps^{1-\theta/2}\iiint_S v\cdot\na_x\chi_\eps r_\eps \dd{x}\dd{v}\dd{t} \\
	&= -\eps^{-\theta}\iiint_S \nu F_\eps(\chi_\eps-\phi)\dd{x}\dd{v}\dd{t}
	- \eps^{1-\theta/2}\iiint_S v\cdot\na_x\chi_\eps r_\eps \dd{x}\dd{v}\dd{t}.
\end{align*}
Since $1-\theta/2>0$, the last term converges (formally) to zero. 
By \cite[Prop.~4.2]{AcMe17} and $E_\eps\to E$, also the first term converges:
\begin{align*}
  -\eps^{-\theta}\iiint_S & \nu(v)F(v,\eps^{\theta-1}E_\eps(x,t))
	(\chi_\eps-\phi) \dd{x}\dd{v}\dd{t} \\
	&\to \int_0^T\int_{\R^d}\big(\kappa(-\Delta)^{\theta/2}\phi 
	+ (DE)\cdot\na_x\phi\big)\rho \dd{x}\dd{t},
\end{align*}
where $\kappa>0$ depends only on $d$, $\gamma$, $\theta$, and the behavior 
of $\nu(v)$ as $|v|\to\infty$, and 
$$
  D = \int_{\R^d}\lambda(v)\otimes v\dd{v}, \quad\mbox{where }Q(\lambda)=\na_vM.
$$

Therefore, integrating by parts, the limit $\eps\to 0$ in \eqref{a.feps2} leads to
$$
  \int_0^T\int_{\R^d} \pa_t\rho\phi \dd{x}\dd{t}
	= \int_0^T\int_{\R^d}\big(-\kappa(-\Delta)^{\theta/2}\rho  
	+ \diver(DE\rho)\big)\phi \dd{x}\dd{t},
$$
which is the weak formulation of the first equation in \eqref{1.eq} since $E=\na_x\psi$.

Finally, we remark that when the scattering rate is constant, $\sigma(v,v')=1$,
the collision operator simplifies to $Q(f)=M\int_{\R^d}f\dd{v}-f$, and the unique
solution to $Q(\lambda)=\na_vM$, $\int_{\R^d}\lambda \dd{v}=0$ equals $\lambda=-\na_vM$.
Consequently, $D_{ij}=0$ for $i\neq j$ and $D_{ii}=-\int_{\R^d}\pa_{v_i}Mv_i\dd{v}
=\int_{\R^d}M\dd{v}=1$, so $D$ equals the unit matrix.


\section{General drift terms}\label{sec.drift}

The drift-diffusion equation of \cite{AcMe17} is of the form (see Appendix 
\ref{sec.deriv})
\begin{equation}\label{a.eq}
  \pa_t\rho + \kappa(-\Delta)^{\theta/2}\rho = \diver(DE\rho)\quad\mbox{in }\R^d,
\end{equation}
where $\kappa>0$ and $D\in\R^{d\times d}$ is a (constant) matrix.
We suppose that $E=\na\psi$ and the potential $\psi$ is a solution to the Poisson
equation $-\Delta\psi=\rho$. In this section, we will illustrate that our
method can be applied only when $D=aI+B$, where $a>0$ and $B$ is skew-symmetric.

This restriction appears when deriving $L^q(\R^d)$ estimates for $\rho(t)$.
Indeed, multiplying \eqref{a.eq} by $\rho^{q-1}$ for some $q>1$ and using
the Stroock-Varopoulos inequality as in Appendix \ref{sec.estrho}, we obtain
\begin{align}
  \frac{1}{q}\frac{\dd{}}{\dd{t}}\|\rho(t)\|_{L^q(\R^d)}^q
	+ \frac{2\kappa}{q}\|(-\Delta)^{\theta/4}(\rho^{q/2})\|_{L^2(\R^d)}^2
	&\leq -\frac{q-1}{q}\int_{\R^d}\na(\rho^{q})\cdot(D\na\psi) \dd{x} \nonumber \\
	&= \frac{q-1}{q}\int_{\R^d}\rho^{q}\diver(D\na\psi)\dd{x}. \label{a.q}
\end{align}
If $D$ equals the unit matrix, the last integral has a sign:
$$
  \int_{\R^d}\rho^{q}\diver(D\na\psi)\dd{x} = \int_{\R^d}\rho^{q}\Delta\psi \dd{x}
	= -\int_{\R^d}\rho^{q+1} \dd{x} \le 0.
$$
For general matrices $D$, we argue as follows. 
Let $R=\na(-\Delta)^{-1/2}$ be the Riesz transform, which can be also
characterized as a Fourier multiplier, $Rf = \F^{-1}[-i(\xi/|\xi|)\widehat f]$.
It has the property $R\cdot R=-I$. 
Then, since $D$ is a constant matrix and $\psi=(-\Delta)^{-1}\rho$,
\begin{align*}
  \diver(D\na\psi) 
	&= \na\cdot(D\na(-\Delta)^{-1}\rho)
	= \na(-\Delta)^{-1/2}\cdot(D\na(-\Delta)^{-{1/2}}\rho) \\
	&= R\cdot(DR\rho)
	= \frac{1}{d}\operatorname{tr}(D) (R\cdot R)\rho 
	-R\cdot\bigg(\frac{1}{d}\operatorname{tr}(D)-D\bigg)R\rho.
\end{align*}

We claim that the last term can be written as
$$
  R\cdot\bigg(\frac{1}{d}\operatorname{tr}(D)-D\bigg)R\rho 
	= c_d\int_{\R^d}\bigg(\frac{1}{d}
	\operatorname{tr}(D)|y|^2 - y\cdot(Dy)\bigg)\frac{\rho(x-y)}{|y|^{d+2}}\dd{y},
$$
for some constant $c_d>0$ only depending on the dimension $d$. 
Set $P(\xi)=\operatorname{tr}(D)|\xi|^2/d + \xi\cdot(D\xi)$.
Since $R$ is a Fourier multiplier, we have
\begin{align*}
  R\cdot\bigg(\frac{1}{d}\operatorname{tr}(D)-D\bigg)R\rho 
	&= -\F^{-1}\bigg[\bigg(\frac{1}{d}\operatorname{tr}(D) + \frac{\xi}{|\xi|}\cdot
	D\frac{\xi}{|\xi|}\bigg)\widehat\rho\bigg]
	= -\F^{-1}\bigg[\frac{P(\xi)}{|\xi|^2}\widehat\rho\bigg] \\
	&= -\F^{-1}\bigg[c_d\F\bigg[\frac{P(x)}{|x|^{d+2}}\bigg]\F[\rho]\bigg]
	= -c_d\frac{P(\xi)}{|\xi|^{d+2}}*\rho \\
	&= -c_d\int_{\R^d}\frac{P(\xi)}{|\xi|^{d+2}}\rho(x-\xi)\dd{\xi},
\end{align*}
where we used $c_d\F[P(x)/|x|^{d+2}]=P(\xi)/|\xi|^2$
(see \cite[Theorem 5, Section III.3]{Ste70}) and the Fourier convolution formula.
This shows that \eqref{a.q} can be written as
\begin{align}
  \frac{1}{q}\frac{\dd{}}{\dd{t}}&\|\rho(t)\|_{L^q(\R^d)}^q
	+ \frac{2\kappa}{q}\|(-\Delta)^{\theta/4}(\rho^{q/2})\|_{L^2(\R^d)}^2 
	\leq -\frac{q-1}{q}\frac{\operatorname{tr}(D)}{d}\int_{\R^d}\rho^{q+1} \dd{x} 
	\nonumber \\
	&{}- \frac{q-1}{q}c_d\int_{\R^d}\int_{\R^d} \rho^q (x) \bigg(\frac{1}{d}
	\operatorname{tr}(D)|y|^2 - y\cdot(Dy)\bigg)\frac{\rho(x-y)}{|y|^{d+2}}\dd{y}\dd{x}.
	\label{c.ineq}
\end{align}

If $\operatorname{tr}(D)\ge 0$, the first integral on the right-hand side 
is nonpositive. The second integral is nonpositive if $y\cdot Dy
\le \operatorname{tr}(D)|y|^2/d$ for all $y\in\R^d$. 
However, Lemma \ref{lem.aux} below shows that under this condition, $D$ can be
decomposed as $D=aI+B$, where $a>0$ and $B$ is skew-symmetric. Consequently,
we are not able to treat general drift matrices $D$.
Clearly, even if the second integral on the right-hand side of \eqref{c.ineq}
is positive, it may happen that it is absorbed by the first integral
such that the right-hand side is still nonpositive, but we are not able to prove this.

\begin{lemma}\label{lem.aux}
Let $D=(D_{ij})\in\R^{d\times d}$ be a matrix such that 
$x\cdot Dx\le \operatorname{tr}(D)|x|^2/d$ for all $x\in\R^d$. 
Then there exists $a>0$ and a skew-symmetric matrix such that 
$D=a\mathrm{I}+B$.
\end{lemma}

\begin{proof}
We can write $D=A+B$, where $A=(D+D^T)/2$ is symmetric and $B=(D-D^T)/2$ is
skew-symmetric. We show that $x\cdot Ax=a|x|^2$ for all $x\in\R^d$ for some $a>0$. 
Since $A$ is symmetric and real, we may assume (by the spectral theorem),
without loss of generality, that $A$ is a diagonal matrix, 
$A=\operatorname{diag}(a_1,\ldots,a_d)$. Furthermore, by homogeneity, it is
sufficient to show this result for all $x\in\R^d$ with $|x|=1$. Then 
$x\cdot Dx\le \operatorname{tr}(D)/d$ is equivalent to
$$
  \sum_{i=1}^d a_ix_i^2 \le \frac{1}{d}\sum_{j=1}^d a_j
$$
or, since $\sum_{i=1}^d x_i^2=1$,
$$
  \sum_{i=1}^d\bigg(a_i - \frac{1}{d}\sum_{j=1}^d a_j\bigg)x_i^2 \le 0.
$$
Choosing $x_i=\delta_{ik}$, we see that
$a_i - (1/d)\sum_{j=1}^d a_j \le 0$,
but this is only possible if $a:=a_i=a_j$ for all $i\neq j$. 
This proves the lemma.
\end{proof}

\end{appendix}


\end{document}